\documentclass[hidelinks,onefignum,onetabnum]{siamart250211}

\usepackage{lipsum}
\usepackage{amsfonts}
\usepackage{graphicx}
\usepackage{epstopdf}
\usepackage{algorithmic}
\usepackage{microtype}
\usepackage{cite}
\ifpdf
\DeclareGraphicsExtensions{.eps,.pdf,.png,.jpg}
\else
\DeclareGraphicsExtensions{.eps}
\fi

\newsiamremark{remark}{Remark}
\newsiamremark{hypothesis}{Hypothesis}
\crefname{hypothesis}{Hypothesis}{Hypotheses}

\newsiamthm{claim}{Claim}
\newsiamremark{fact}{Fact}
\crefname{fact}{Fact}{Facts}

\allowdisplaybreaks[4]
\usepackage{mathtools}
\usepackage{enumerate}
\usepackage{booktabs}
\usepackage{subfigure}
\usepackage[percent]{overpic}

\usepackage{multirow}
\newcommand{\tabincell}[2]{\begin{tabular}{@{}#1@{}}#2\end{tabular}}

\newsiamremark{assumption}{Assumption}
\crefname{assumption}{Assumption}{Assumptions}
\newsiamremark{example}{Example}
\crefname{example}{Example}{Examples}
\def\R{{\mathbb{R}}}
\def\tr{{\rm tr}}

\newcommand{\changescolor}{\color{black}}

\headers{A smoothing moving balls approximation method}{J. Xu, T. K. Pong{\and} N.-S. Sze}

\title{A smoothing moving balls approximation method for a class of conic-constrained difference-of-convex optimization problems\thanks{Submitted to the editors \today.
		\funding{Ting Kei Pong is supported partly by the Hong Kong Research Grants Council PolyU 15300423. Nung-sing Sze is supported partly by the Hong Kong Research Grants Council PolyU 15300121 and a PolyU research grant 4-ZZRN.}}}

\author{Jiefeng Xu\thanks{Department of Applied Mathematics, the Hong Kong Polytechnic University, Hong Kong, People's Republic of China
		(\email{jiefeng.xu@polyu.edu.hk}).}
\and Ting Kei Pong\thanks{Department of Applied Mathematics, the Hong Kong Polytechnic University, Hong Kong, People's Republic of China (\email{tk.pong@polyu.edu.hk}).}
\and Nung-Sing Sze\thanks{Department of Applied Mathematics, the Hong Kong Polytechnic University, Hong Kong, People's Republic of China (\email{raymond.sze@polyu.edu.hk}).}
}

\usepackage{amsopn}

\ifpdf
\hypersetup{
	pdftitle={An Example Article},
	pdfauthor={D. Doe, P. T. Frank, and J. E. Smith}
}
\fi

\begin{document}
	
	\maketitle
	
	\begin{abstract}
		In this paper, we consider the problem of minimizing a difference-of-convex objective over a nonlinear conic constraint, where the cone is closed, convex, pointed and has a nonempty interior. We assume that the support function of a compact base of the polar cone exhibits a {\em majorizing smoothing approximation}, a condition that is satisfied by widely studied cones such as $\R^m_-$ and ${\cal S}^m_-$. Leveraging this condition, we reformulate the conic constraint equivalently as a {\em single} constraint involving the aforementioned support function, and
adapt the moving balls approximation (MBA) method for its solution. In essence, in each iteration of our algorithm, we approximate the support function by a smooth approximation function and apply one MBA step. The subproblems that arise in our algorithm always involve only one {\em single} inequality constraint, and can thus be solved efficiently via {\em one-dimensional} root-finding procedures. We design explicit rules to evolve the smooth approximation functions from iteration to iteration and establish the corresponding iteration complexity for obtaining an {\changescolor $(\epsilon_1, \epsilon_2, \epsilon_1 \sqrt{\epsilon_2})$}-Karush-Kuhn-Tucker point. In addition, in the convex setting, we establish convergence of the sequence generated, and study its local convergence rate under a standard H\"olderian growth condition. Finally, we {\changescolor perform numerical experiments to illustrate the performance of our algorithm.}
	\end{abstract}
	
	\begin{keywords}
		Conic-constrained optimization, moving balls approximation, smoothing, difference-of-convex optimization
	\end{keywords}
	
	\begin{MSCcodes}
		65Y20, 90C06, 90C26, 90C55
	\end{MSCcodes}
	
	\section{Introduction}
	Nonlinear conic programming plays a central role in a wide range of modern applications, including robust optimal control \cite{98BCH}, portfolio optimization \cite{98LVBL}, and structural optimization \cite{04KS}, etc. We refer the readers to \cite{11AL, 12WSV, 15YY} and the references therein for more discussions.

	In this paper, we consider the following conic-constrained difference-of-convex (DC) minimization
	problem
	\begin{equation}\label{opt-conic-dc}
		\begin{array}{rl}
			\min\limits_{x\in {\mathbb X}}& \psi(x) \coloneqq  f(x) + P_1(x) - P_{2}(x)
			\\ \text{s.t.}& G(x) \in \mathcal{K},
		\end{array}
	\end{equation}
	where ${\cal K}\subseteq \mathbb{Y}$, ${\mathbb X}$ and ${\mathbb Y}$ are finite dimensional real Hilbert spaces,
	and $f$, $P_1$, $P_2$, $G$ and ${\cal K}$ satisfy the following assumption (see section~\ref{sec2} for notation).
	\begin{assumption}\label{ass-gen} In problem \eqref{opt-conic-dc},
		\begin{enumerate}[{\rm (i)}]
			\item the set $\mathcal{K} \subseteq {\mathbb Y}$ is a closed convex pointed cone with a nonempty interior;
			\item the function $f:{\mathbb X}\rightarrow \mathbb{R}$ and the mapping
			$G:{\mathbb X}\rightarrow {\mathbb Y}$ are $L_{f}$-smooth and $L_G$-smooth with constants $L_f>0$ and $L_G\ge 0$, respectively;			
			\item the functions $P_1:\mathbb{X}\to \mathbb{R}$ and $P_{2}:{\mathbb X}\rightarrow \mathbb{R}$ are convex;			
			\item there exists an $x^0\in {\mathbb X}$ such that $G(x^{0}) \in \mathrm{int}(\mathcal{K})$ and the set $\Omega_{0}\coloneqq  \{x \in {\mathbb X}: \psi(x) \leq \psi(x^0),\, G(x)\in
			\mathcal{K}\}$ is bounded;
			\item the Robinson's Constraint Qualification (RCQ) holds in $\Omega_{0}$, i.e.,
			\begin{equation*}
				0 \in \mathrm{int}\left( G(x) + DG(x) {\mathbb X} - \mathcal{K}\right)\ \ \ \forall x\in \Omega_0.
			\end{equation*}
		\end{enumerate}
	\end{assumption}	
Notice that under Assumption~\ref{ass-gen}, the solution set of \eqref{opt-conic-dc} is nonempty.

Model \eqref{opt-conic-dc} captures many important instances such as nonlinear programs, nonlinear semidefinite programs {\changescolor (NSDPs)}, and nonlinear second order cone programs, etc. Various algorithms have been developed to address such optimization problems, including interior point methods (see, e.g., \cite{94NN,00Jarre,15YY,23AOT,23HL}), augmented Lagrangian based methods (see, e.g., \cite{69Powell,76Rockafellar,23LZ,21Xu,18FL,08SSZ}) and methods based on sequential quadratic programming methods (see, {\changescolor e.g., \cite{09Solodov,10FS,13A,21FKLS,07KF,04CR}}), etc. 
Here, we are interested in feasible methods (i.e., all iterates satisfy $G(x)\in {\cal K}$) that invoke only first-order information such as $\nabla f$, $DG$, $\partial P_2$, the proximal mapping of $\gamma P_1$ for some $\gamma > 0$, and possibly the projection onto ${\cal K}$.

When $\mathcal{K} = \mathbb{R}^m_{-}$, a representative class of feasible first-order methods is the moving balls approximation (MBA) method and its variants \cite{10AST,16BP,16ST,20BCP,21YPL}. The MBA method was proposed in \cite{10AST} for \eqref{opt-conic-dc} with ${\cal K} = \mathbb{R}^m_{-}$ and $P_1=P_2 = 0$. The key idea is to replace $f$ and each component function of $G$ by a quadratic majorant (with the Hessian being a multiple of identity) at the current iterate, forming a special quadratically constrained quadratic programming subproblem with $m$ inequality constraints. Convergence of the whole sequence was established under suitable assumptions such as convexity or the Kurdyka-{\L}ojasiewicz property \cite{10AST,16BP}. Later, linesearch techniques were incorporated in \cite{20BCP,21YPL}, and an MBA variant was developed in \cite{21YPL} for \eqref{opt-conic-dc} with ${\cal K} = \mathbb{R}^m_{-}$. In the aforementioned works, the MBA subproblems were assumed to be solved {\em exactly} in their convergence analysis. This is an innocuous assumption when $m = 1$ and the proximal mapping of $\gamma P_1$ can be computed efficiently for all $\gamma > 0$ since the subproblems can then be solved efficiently via some {\em one-dimensional} root-finding procedures.\footnote{\changescolor More importantly, these subproblems admit closed form formulae when $P_1 = 0$, and can be solved efficiently and {\em exactly} when $P_1$ is the $\ell_1$ norm, which corresponds to finding roots of {\em piecewise linear quadratic} equations; see \cite[Appendix~A]{21YPL}.} However, when $m > 1$, the corresponding MBA subproblems typically require an iterative solver (see, e.g., \cite[section~6]{10AST}), and the inexactness in solving the subproblems has to be taken into account for a better understanding of the practical behavior of MBA and its variants.

The two recent works that considered MBA variants with inexactly solved subproblems are \cite{25BDL} and \cite{25LPB}. In \cite{25BDL}, they considered \eqref{opt-conic-dc} with $P_2 = 0$, ${\cal K} = \R^m_-$ and a more general $G$ that is allowed to be nonsmooth. Supposing that the subproblems of their algorithm are solved up to an inexact criterion as in \cite[Definition~7]{25BDL}, they derived the {\em overall} computational complexity of their method, including the cost for inexactly solving the subproblems, for obtaining an $\epsilon$-Karush-Kuhn-Tucker ($\epsilon$-KKT) point; see Corollaries 3, 4 and 5 of \cite{25BDL}. More recently, carefully designed implementable termination criteria for their MBA subproblems were proposed in \cite{25LPB} for \eqref{opt-conic-dc} with ${\cal K} = \R^m_-$, and the iteration complexity (of the outer loop) and asymptotic convergence properties were established. However, it is not immediately clear whether these termination criteria can be readily extended to deal with \eqref{opt-conic-dc} for more general cones. Recalling that the MBA subproblems for \eqref{opt-conic-dc} with ${\cal K}=\R^m_-$ can be solved efficiently when $m = 1$ and the proximal mapping of $\gamma P_1$ can be computed efficiently for all $\gamma > 0$, can we {\em enhance} the MBA method for \eqref{opt-conic-dc} so that all the subproblems that arise will involve only {\em one} single inequality constraint, at least for some special class of cones ${\cal K}$?

Our point of departure is the fact that for a closed convex pointed cone that has a nonempty interior,  its polar cone
	has a compact base;\footnote{A base for a cone $\mathcal{Q}\subseteq \mathbb{Y}$ is a convex set $\mathcal{D}$ with $0\notin
		\mathrm{cl}(\mathcal{D})$ and $\mathcal{Q} = \cup_{t\ge 0} t\mathcal{D}$.} see, e.g., {\cite[Page~60, Exercise 20]{06BL}}. Leveraging this observation,
we reformulate \eqref{opt-conic-dc} equivalently as the following single-constrained optimization problem:
	\begin{equation}\label{opt-supp-dc}
		\begin{array}{rl}
			\min\limits_{x\in {\mathbb X}}& \psi(x) \coloneqq  f(x) + P_1(x) - P_{2}(x)
			\\ \text{s.t.}& g_{\mathcal{B}}(x)\coloneqq \sigma_{\mathcal{B}}(G(x)) \le 0,
		\end{array}
	\end{equation}
	where $\mathcal{B}$ is a compact base of $\mathcal{K}^{\circ}$ {\changescolor and $\sigma_{\mathcal{B}}$ is the support function of $\mathcal B$ (see  section~2.1 for the definition of $\sigma_{\cal B}$)}.\footnote{The equivalence can be seen as an immediate consequence of the last part of Lemma~\ref{lm-base} below.} Our next key assumption is that $\sigma_{\cal B}$ exhibits a family of {\em majorizing smoothing approximation} $\{h_{\mu}\}_{\mu>0}$: this assumption is a specialization of \cite[Definition~2.1]{12BT}, and is satisfied by commonly studied ${\cal K}$ such as $\R^m_-$ and ${\cal S}^m_-$, see Remark~\ref{rem25}(i) below. {\changescolor Leveraging} this assumption {\changescolor and a mild descent assumption (see \eqref{eq-c4})}, we adapt the MBA method for solving \eqref{opt-supp-dc} (and hence, equivalently, \eqref{opt-conic-dc}), which we call the smoothing moving balls approximation method ({\em s}MBA). In each iteration of {\em s}MBA, we approximate $\sigma_{\cal B}$ by a smooth function $h_{\mu_k}$ from the family $\{h_\mu\}_{\mu > 0}$ and apply {\em one} MBA step: this results in subproblems that involve only one {\em single} inequality constraint, which can be solved efficiently via {\changescolor simple} {\em one-dimensional} root-finding procedures as long as the proximal mapping of $\gamma P_1$ can be computed efficiently for all $\gamma > 0$; see, e.g., \cite[Appendix~A]{21YPL} for the case when $P_1$ is the $\ell_1$ norm {\changescolor where such subproblems can be solved efficiently and even {\em exactly}}. 
    Our approach is in contrast with recent
works \cite{25BDL,25LPB} that studied the MBA method applied to \eqref{opt-conic-dc} with ${\cal K} = \R^m_-$, whose subproblems involve possibly multiple inequality constraints, and necessitate iterative subproblem solvers and a careful design of inexact termination criteria for these solvers. We design explicit rules for picking $\{\mu_k\}$ in {\em s}MBA, and establish the corresponding iteration complexity for finding an {\changescolor $(\epsilon_1, \epsilon_2, \epsilon_1 \sqrt{\epsilon_2})$}-KKT point of \eqref{opt-conic-dc}. Furthermore, in the convex setting (i.e., $f$ is convex, $P_2 = 0$ and $G$ is $(-{\cal K})$-convex), we present rules for choosing $\{\mu_k\}$ to guarantee convergence of the whole sequence generated by {\em s}MBA, and study its local convergence rate under a standard H\"olderian growth condition. Finally, we {\changescolor perform numerical experiments to illustrate the performance of $s$MBA}. 
	
The remainder of the paper is organized as follows. We present notation and discuss preliminary materials in section~\ref{sec2}. Our algorithm is presented in section~\ref{sec:alg} and its iteration complexity is studied in section~\ref{sec:complexity}. Sequential convergence and local convergence rate in the convex setting are studied in section~\ref{sec-conv}. {\changescolor Numerical experiments are presented in section~\ref{sec5}. We give some concluding remarks in section~6.}

\section{Notation and preliminaries}\label{sec2}
	\subsection{Notation}
In this paper, we use $\mathbb{X}$ and $\mathbb{Y}$ to denote two finite-dimensional real Hilbert spaces, and we use $\langle \cdot, \cdot \rangle$ and $\| \cdot \|$ to denote the inner product and the associated norm, respectively, of the underlying Hilbert space.
	We let $\mathbb{R}$ denote the set of real numbers and $\mathbb{R}^n$ denote the $n$-dimensional Euclidean space.
	We also let $\mathbb{R}^{n}_{+}$ (resp., $\mathbb{R}^{n}_{-}$) denote the non-negative (resp., non-positive) orthant in $\mathbb{R}^{n}$.
	The set of $m\times n$ matrices is denoted by $\mathbb{R}^{m \times n}$. For two matrices $A, B \in \mathbb{R}^{m \times n}$, their
	(trace) inner product is defined as $\langle A, B \rangle {\changescolor\, \coloneqq} \mathrm{tr}(A^\top B)$, where $\mathrm{tr}$ denotes the trace of a square matrix.
	We also let $\mathcal{S}^n$ denote the space of $n \times n$ symmetric matrices, and let $\mathcal{S}_+^n$ (resp., $\mathcal{S}_{-}^{n}$) denote the cone of $n \times n$ positive (resp., negative) semidefinite matrices. The set of natural numbers is denoted by $\mathbb{N}$ and we write $\mathbb{N}_0 \coloneqq \mathbb{N}\cup\{0\}$.

	We denote the closed ball in $\mathbb{X}$ with center $x\in {\mathbb X}$ and radius $r\ge 0$ by $B(x, r) {\changescolor\, \coloneqq} \{z\in {\mathbb X}: \|z - x\|\le r\}$. For a set $\mathcal{C}\subseteq \mathbb{X}$, we denote the distance from any $x\in \mathbb{X}$ to ${\cal C}$ as $\mathrm{dist}(x, \mathcal{C}) {\changescolor\, \coloneqq} \inf_{z\in \mathcal{C}} \|x-z\|$. The convex hull, closure and interior of ${\cal C}$ are denoted, respectively, by $\mathrm{conv}({\cal C})$, ${\rm cl}({\cal C})$ and ${\rm int}({\cal C})$.
	The support function of ${\cal C}$ is denoted by $\sigma_{\mathcal{C}}$, which is defined as $\sigma_{\mathcal{C}}(x) {\changescolor\, \coloneqq} \sup_{z\in \mathcal{C}}\langle x, z \rangle$ for all $x\in\mathbb{X}$.
	A cone $\mathcal{Q}\subseteq \mathbb{Y}$ is a convex set that contains $0$ and satisfies $\alpha {\cal Q}\subseteq {\cal Q}$ for all $\alpha \ge 0$. A cone ${\cal Q}$ is said to be pointed if $\mathcal{Q} \cap -\mathcal{Q} = \{0\}$.
	For a cone $\mathcal{Q}\subseteq \mathbb{Y}$, its polar is defined as $\mathcal{Q}^{\circ} {\changescolor\, \coloneqq} \{y\in \mathbb{Y} : \langle y, w \rangle \le 0\ \ \forall w\in \mathcal{Q} \}$.

For a convex function $\varphi\colon \mathbb{X}\to \mathbb{R}$, its subdifferential at $x\in \mathbb{X}$ is defined as
	\begin{equation*}
		\partial \varphi(x) {\changescolor\, \coloneqq} \{\xi\in \mathbb{X} : \varphi(z) \ge \varphi(x) + \langle \xi, z-x \rangle\quad \forall z \in \mathbb{X} \}.
	\end{equation*}
	In addition, for every $\epsilon\ge 0$, the $\epsilon$-subdifferential of $\varphi$ at $x\in \mathbb{X}$ is defined as
	\begin{equation*}
		\partial_{\epsilon} \varphi(x) {\changescolor\, \coloneqq} \{\xi \in \mathbb{X} : \varphi(z) \ge \varphi(x) + \langle \xi, z-x \rangle - \epsilon \quad \forall z \in \mathbb{X} \}.
	\end{equation*}
Clearly, it holds that $\partial_0 \varphi(x) = \partial \varphi(x)$.
	
	For a mapping $H: \mathbb{X}\to \mathbb{Y}$, we say that $H$ is $L$-Lipschitz continuous if there exists $L\ge 0$ such that
	$$
	\| H(x) - H(z) \| \le L \|x-z\| \quad \forall x,z\in \mathbb{X}.
	$$
	For a continuously differentiable mapping $H$, we use $DH(x)$ to denote its derivative map at $x\in \mathbb{X}$: this is the linear map defined by
	$$
	DH(x)\,d
	{\changescolor\, \coloneqq}
	\lim_{t\to0}\frac{H(x+t\,d)-H(x)}{t}
	\ \ \ \
	\forall\,d\in\mathbb{X}.
	$$
	The adjoint of $DH(x)$ is denoted by $DH(x)^*$.
	A mapping $H$ is said to be $L_{H}$-smooth ($L_H\ge 0$) if it is continuously differentiable with its derivative $DH$ being $L_H$-Lipschitz continuous with constant $L_H$.
	Finally, we say that $H$ is $\mathcal{Q}$-convex for a cone $\mathcal{Q}\subseteq \mathbb{Y}$ if
	$$
	\lambda H(x) + (1-\lambda) H(z)\in H(\lambda x + (1-\lambda)z) + \mathcal{Q} \quad \forall x,z\in \mathbb{X},\ \lambda \in [0, 1].
	$$


	\subsection{Properties of pointed cone}
	In Assumption~\ref{ass-gen}, $\mathcal{K} \subset {\mathbb Y}$ is a closed convex pointed cone
	with nonempty interior.
	It follows from \cite[Page~60, Exercise 20]{06BL} that $\mathcal{K}^{\circ}$ is pointed, has a nonempty interior, and possesses a compact base $\mathcal{B}$: recall that a base for a cone $\mathcal{Q}\subseteq{\mathbb{Y}}$ is a convex set $\mathcal{D}$ with $0\notin
		\mathrm{cl}(\mathcal{D})$ and $\mathcal{Q} = \cup_{t\ge 0} t\mathcal{D}$. We summarize the key properties of ${\cal K}$ necessary for our development in the next lemma.
	\begin{lemma}\label{lm-base}
		Let $\mathcal{K} \subset {\mathbb Y}$ be a closed convex pointed cone
		with nonempty interior. Then, for every $w\in \mathrm{int}(\mathcal{K})$, the set
		$\mathcal{B}_{w} \coloneqq  \{u \in \mathcal{K}^{\circ} :
		\langle w, u\rangle = -1\}$ is a compact base of $\mathcal{K}^{\circ}$.
		Furthermore, for any compact base ${\cal B}$ of ${\cal K}^\circ$, it holds that $y\in \mathcal{K}$ $(\mbox{resp.,
		}y\in\mathrm{int}(\mathcal{K}))$ if and only if $\sigma_{\mathcal{B}}(y)\le 0$
		$(\mbox{resp., }\sigma_{\mathcal{B}}(y) < 0 )$.
	\end{lemma}

We present several examples of bases $\mathcal{B}$ (taking the form of ${\cal B}_w$ for some $w\in {\rm int}({\cal K})$) and their corresponding $\sigma_{\mathcal{B}}$, under various choices of the space $\mathbb{Y}$ and cone $\mathcal{K}$, in Table~\ref{tab-const-ex}.
	
	\begin{table}[htbp]
		\footnotesize
		\caption{\noindent Examples of bases of polar cones and their support functions.
			Each base $\mathcal{B} = \mathcal{B}_{w}$ is constructed from a $w\in {\rm int}(\mathcal{K})$, where $\mathcal{B}_{w}$ is defined in Lemma~\ref{lm-base}.
			Here, we denote the $m\times m$ identity matrix by $I_{m}$, the maximum eigenvalue of $y\in \mathcal{S}^{m}$ by $\lambda_{\mathrm{\max}}(y)$ and the $p$-cone by $\mathcal{K}_{m,p} \coloneqq  \left\{(u, t) \in \mathbb{R}^{m} \times \mathbb{R} : \|u\|_{p} \le t \right\}$ with $\|u\|_p := (\sum_{i=1}^m|u_i|^p)^{1/p}$, $p, q \in (1, \infty)$ and $\frac{1}{p} + \frac{1}{q} = 1$. }\label{tab-const-ex}
		\begin{center}
			\begin{tabular}{cccccc}
				\toprule
				${\mathbb Y}$ &  $\mathcal{K}$ & $\mathcal{K}^{\circ}$ & $w$  &  $\mathcal{B}$ & $\sigma_{\mathcal{B}}(y)$ for $y\in \mathbb{Y}$  \\
				\midrule
				$\mathbb{R}^{m}$ & $\mathbb{R}_{-}^{m}$ & $\mathbb{R}_{+}^{m}$ & $-(1,\ldots,1)$ & $\{u\in \mathbb{R}^{m}_{+} :  \sum_{i=1}^{m}u_i = 1\} $ & $\max_{1\le i\le m}y_i$ \\
				$\mathcal{S}^{m}$ & $\mathcal{S}_{-}^{m}$ & $\mathcal{S}_{+}^{m}$ & $-I_{m}$ & $\{u\in \mathcal{S}^{m}_{+} : \mathrm{tr}(u) = 1\}$ & $\lambda_{\max}(y)$ \\
				$\mathbb{R}^{m+1}$ & $\mathcal{K}_{m,p}$ & $-\mathcal{K}_{m,q}$ & $(0, 1)$ & $\{(u, -1) \in \mathbb{R}^{m+1} :  \|u\|_{q}\le 1\} $ &  $\sqrt[p]{\sum_{i=1}^{m}|y_i|^{p}} - y_{m+1}$ \\
				\bottomrule
			\end{tabular}
		\end{center}
	\end{table}
	
	For a compact base $\mathcal{B}$ of ${\cal K}^\circ$, $\sigma_{\mathcal{B}}$ is convex and real-valued on ${\mathbb Y}$.
	Combining this with \cite[Theorem 2.4.14]{02Zalinescu}, we have for every $\epsilon\ge 0$ that
	\begin{equation}\label{eq-sig-subdiff}
		\partial_{\epsilon} \sigma_{\mathcal{B}}(y)
		= \left\{u \in \mathcal{B}: \langle u, y \rangle \ge \sigma_{\mathcal{B}}(y) -
		\epsilon\right\} \quad \forall\, y\in {\mathbb Y}.
	\end{equation}
	This implies $\partial_{\epsilon} \sigma_{\mathcal{B}}(y) \subseteq
	\mathcal{B}$ for all $\epsilon\ge 0$ and all $y\in{\mathbb Y}$.
	In particular, this implies that $\sigma_{\mathcal{B}}$ is globally Lipschitz
	continuous with constant $M_{\mathcal{B}}>0$ (since $0\notin {\cal B}$) given by
	\begin{equation}\label{eq-M-B-def}
		M_{\mathcal{B}}\coloneqq \sup_{u\in \mathcal{B}}\|u\| < \infty.
	\end{equation}
	
	\subsection{Support function and its majorizing smoothing approximation}
	The following definition of majorizingly smoothable functions is adapted from \cite[Definition~2.1]{12BT}. Specifically, it is obtained by setting $\beta_1=0$ in \cite[Definition 2.1]{12BT}, meaning that we only consider those $h_\mu$ that majorize $h$.
	\begin{definition}
		A convex function $h\!:\!{\mathbb Y}\!\to\! \mathbb{R}$ is said to be $(\alpha_{1}, \alpha_{2}, \alpha_{3})$-majorizingly smoothable for some $\alpha_1\ge 0$, $\alpha_2 > 0$ and $\alpha_3 > 0$, if there exists a family $\{h_{\mu}\}_{\mu > 0}\subset
		C^{1}({\mathbb Y})$, referred to as
		a majorizing smoothing approximation (MSA) of $h$, such that the following conditions hold:
		\begin{enumerate}[{\rm (i)}]
			\item the function $h_{\mu}$ is convex, differentiable, and its gradient is $(\alpha_{1} + \alpha_{2} / \mu)$-Lipschitz continuous, for every $\mu > 0$;
			
			\item it holds that
			\begin{equation}\label{eq-sm-appr-bd}
				h(y) \leq h_{\mu}(y) \leq h(y) + \mu \alpha_{3}\quad
				\forall\, y\in{\mathbb Y},\,\mu>0.
			\end{equation}			
		\end{enumerate}
	\end{definition}
	
	From \eqref{eq-sm-appr-bd}, we see in particular that $\lim_{\mu \downarrow 0}h_\mu(y) = h(y)$ for all $y\in \mathbb{Y}$. Some concrete examples of majorizingly smoothable functions with explicitly defined MSAs are presented in Table~\ref{tab-sm-ex} below.
	
	\begin{table}[htbp]
		\footnotesize
		\caption{\noindent Some majorizingly smoothable functions $h$.
			Example~1 is adapted from \cite[Example 4.4]{12BT};
			Example~2 is based on \cite[Eq.~(10), Eq.~(17)]{07Nesterov} and \cite[Theorem 4.2]{12BT}, where $\lambda_i(y)$ is the $i$th largest eigenvalue of $y$;
			Example~3 is from \cite[Example 4.6]{12BT}. We assume $m\ge 2$ in Examples 1 and 2.}\label{tab-sm-ex}
		\begin{center}
			\begin{tabular}{ccccc}
				\toprule
				Example & ${\mathbb Y}$ & $h(y)$ & $h_{\mu}(y)$ & $(\alpha_1, \alpha_2, \alpha_3)$\\
				\midrule
				1 &  $\mathbb{R}^{m}$ & $\max_{1 \le i \le m}y_i$
				& $\mu \log \left(\sum_{i=1}^{m}e^{y_i/\mu}\right)$
				& $(0, 1,\log(m))$\\
				2 &  $\mathcal{S}^{m}$ & $\max_{1 \le i \le m}\lambda_i(y)$
				& $\mu \log \left(\sum_{i=1}^{m}e^{\lambda_i(y)/\mu}\right)$
				& $(0, 1,\log(m))$\\
				3 &  $\mathbb{R}^{m}$ & $\sum_{i=1}^m |y_i|$
				& $\sum_{i=1}^{m}\sqrt{y_i^2 + \mu^2}$
				& $(0, 1, m)$\\
				\bottomrule
			\end{tabular}
		\end{center}
	\end{table}

	The following proposition describes a relationship between the gradients of an MSA $\{h_\mu\}$ and the $\epsilon$-subdifferential of $h$.
	\begin{proposition}\label{prop-eps-subdiff}
		Let $h
		:{\mathbb Y}\rightarrow \mathbb{R}$ be a convex function.
		If $\{h_{\mu}\}_{\mu>0}$ is an MSA of
		$h$ with parameters $(\alpha_{1},\, \alpha_{2},\, \alpha_{3})$, then $\nabla h_{\mu}(y) \in \partial_{\alpha_3 \mu} h (y)$ for every
		$y\in{\mathbb Y}$ and $\mu>0$.
	\end{proposition}
	\begin{proof}
		For any $\bar y\in{\mathbb Y}$ and $\mu>0$, by the convexity of $h_{\mu}$, we
		have
		\begin{equation}\label{eq-subdiff}
			h_{\mu}(y) \ge h_{\mu}(\bar y) + \langle \nabla h_{\mu}(\bar
			y), y - \bar y\rangle\quad \forall\, y \in {\mathbb Y}.
		\end{equation}
		Using \eqref{eq-sm-appr-bd}, one has $h_{\mu}(y) \le h(y) +
		\alpha_{3} \mu$ and $h(\bar y) \le h_{\mu}(\bar y)$.
		Combining these with \eqref{eq-subdiff} gives
		\begin{equation*}
			h(y) \ge h(\bar y) + \langle \nabla h_{\mu}(\bar y), y - \bar
			y \rangle - \alpha_{3} \mu\quad \forall\, y \in {\mathbb Y},
		\end{equation*}
		which shows that $\nabla h_{\mu}(\bar y) \in \partial_{\alpha_{3}
			\mu}h(\bar y)$.
	\end{proof}
	
	We next introduce an important assumption regarding $\sigma_{\cal B}$ in \eqref{opt-supp-dc}.
	\begin{assumption}\label{ass-sm} The support function $\sigma_{\mathcal{B}}$ from \eqref{opt-supp-dc} admits an MSA $\{h_{\mu}\}_{\mu>0}$ with parameters $(\alpha_1, \alpha_2, \alpha_3)$.\footnote{More precisely, we mean that such $\{h_\mu\}_{\mu > 0}$ exists and we can compute their function values and gradients efficiently.}
		Moreover, there exists $\alpha_4 > 0$ such that
		\begin{equation}\label{eq-c4}
			h_{\mu_1}(y) \leq h_{\mu_0}(y) - \alpha_4 (\mu_0 - \mu_1)\quad \forall\,
			y\in{\mathbb Y},\, \mu_{0}>\mu_1>0.
		\end{equation}
	\end{assumption}
\begin{remark}[Remark on Assumption~\ref{ass-sm}]\label{rem25}
  \begin{enumerate}[{\rm (i)}]
    \item Table~\ref{tab-sm-ex} lists some examples of $\sigma_{\cal B}$ that admit MSAs. Indeed, the first two $h$'s there correspond to $\sigma_{\cal B}$ with ${\cal B}$ equals $\{u\in \R^m_+:\sum_{i=1}^m u_i = 1\}$ and $\{u\in {\cal S}_+^m:\; \tr(u)=1\}$, respectively. These correspond to ${\cal K}=\R^m_-$ and ${\cal S}^m_-$ respectively, according to Table~\ref{tab-const-ex}.
    \item It should be noted that the additional condition \eqref{eq-c4} is quite mild. Indeed, one can always modify an MSA that is {\em nondecreasing} (in the sense that $h_{\mu_1}\le h_{\mu_0}$ pointwise whenever $\mu_0\ge \mu_1>0$) to satisfy this condition.
	To see this, let $\{\bar h_{\mu}\}_{\mu>0}$ be an MSA of $\sigma_{\mathcal{B}}$ with parameters $(\alpha_1, \alpha_2, \alpha_3)$
	that satisfies additionally $\bar h_{\mu_1}\le \bar h_{\mu_0}$ for any $\mu_0\ge \mu_1>0$: one can check that this additional assumption is satisfied by each MSA in Table~\ref{tab-sm-ex}.
	Now, let $\alpha_4 > 0$ and define $\{h_{\mu}\}_{\mu>0}$ by
	\begin{equation*}
		h_{\mu}(y) {\changescolor\coloneqq}  \bar h_{\mu}(y) + \alpha_4 \mu \quad \forall\, y\in {\mathbb Y}, \mu>0.
	\end{equation*}
	Then one can see that $\{h_{\mu}\}_{\mu>0}$ is an MSA of $\sigma_{\mathcal{B}}$ with parameters $(\alpha_1, \alpha_2, \alpha_3+\alpha_4)$.
	Moreover, for all $y\in {\mathbb Y}$ and any $\mu_0>\mu_1>0$, we have
	\begin{equation*}
		h_{\mu_1}(y) = \bar h_{\mu_1}(y) + \alpha_4 \mu_1
		\le \bar h_{\mu_0}(y) + \alpha_4 \mu_1
		= h_{\mu_0}(y) - \alpha_4 (\mu_{0} - \mu_1),
	\end{equation*}
showing that $\{h_{\mu}\}_{\mu>0}$ also satisfies \eqref{eq-c4}.
  \end{enumerate}
\end{remark}
		
	We now consider smoothing the composite function $g_{\mathcal{B}}$ in \eqref{opt-supp-dc} by smoothing $\sigma_{\cal B}$. The next proposition discusses some basic properties of the resulting function.
	\begin{proposition}\label{prop-g-mu}
		Consider \eqref{opt-conic-dc} and \eqref{opt-supp-dc}, and suppose
		that Assumptions~\ref{ass-gen} and \ref{ass-sm} hold.
		For each $\mu>0$, define $g_{\mu} : {\mathbb X} \to \mathbb{R}$ by
		\begin{equation}\label{eq-g-mu-def}
			g_{\mu}(x) {\changescolor\coloneqq}  h_{\mu}(G(x)) \quad \forall\, x\in {\mathbb X}.
		\end{equation}
		Let $M_{\mathcal{B}}$ be defined in \eqref{eq-M-B-def}. Then the following statements hold:
		\begin{enumerate}[{\rm (i)}]
			\item We have
			\begin{align}
				& g_{\mathcal{B}}(x)\le g_{\mu}(x)\le g_{\mathcal{B}}(x) + \alpha_{3} \mu\quad \forall\, x\in {\mathbb X},\,
				\mu>0, \label{eq-gmu-appr}\\
				& g_{\mu_1}(x) \le g_{\mu_{0}}(x) - \alpha_{4} (\mu_{0} - \mu_1)\quad \forall\,
				x\in {\mathbb X},\, \mu_{0}>\mu_1>0. \label{eq-gmu-str-des}
			\end{align}
			
			\item
			It holds that\footnote{In \eqref{eq-M-G-def}, $\|DG(x)\|:= \sup\{\|DG(x) d\| : \|d\|=1\}$ is the operator norm.}
			\begin{align}
				& \|\nabla g_{\mu} (x) \| \le M_{G}M_{\mathcal{B}}< \infty \quad \forall\, x\in
				\Omega_{0},\,\mu>0, \label{eq-bd-nabla g_mu}\\
				& M_{G} \coloneqq  \sup\{ \| DG(x)\|: x \in \Omega_{0}\}<\infty. \label{eq-M-G-def}
			\end{align}
			\item It holds that
			\begin{equation}\label{eq-low-quad-bd}
				g_{\mu}(x) \ge g_{\mu}(\bar x) + \langle \nabla g_{\mu}(\bar x), x - \bar
				x\rangle - \frac{L_{G}M_{\mathcal{B}}}{2}\|x - \bar x\|^2
				\quad \forall\, x,\, \bar x\in {\mathbb X},\, \mu>0.
			\end{equation}
			
			\item If the mapping $G$ is $(-\mathcal{K})$-convex, then $g_{\mu}$ is convex for every $\mu>0$.
		\end{enumerate}
	\end{proposition}
	\begin{proof}
		Item (i) follows directly from \eqref{eq-sm-appr-bd}, \eqref{eq-c4} and the definition of $g_{\mu}$ in \eqref{eq-g-mu-def}.
		We now turn to item (ii).
		By Assumption~\ref{ass-gen}, $\Omega_{0}$ is bounded, which shows $M_G<\infty$.
		From Proposition~\ref{prop-eps-subdiff} and \eqref{eq-sig-subdiff}, we have, for any $x\in{\mathbb X}$ and $\mu>0$,
		$$
		\nabla h_{\mu}(G(x)) \in \partial_{\alpha_3 \mu}\sigma_{{\cal B}}(G(x)) \subseteq
		\mathcal{B},
		$$
		which, together with $\nabla g_{\mu}(x) = DG(x)^* \nabla h_{\mu}(G(x))$, shows that for any $x\in \Omega_{0}$,
		$$
		\|\nabla g_{\mu}(x)\| \le \|DG(x)\|\|\nabla h_{\mu}(G(x))\| \le
		M_{G}M_{\mathcal{B}}<\infty.
		$$
		
		We now prove item (iii). Let $\mu>0$ and $x,\, \bar x\in {\mathbb X}$.
		Then it holds that
		\begin{align*}
				&\langle \nabla h_{\mu}(G(\bar x)), G(x) - G(\bar x) \rangle
				\\=&
				\langle \nabla h_{\mu}(G(\bar x)),  DG(\bar x)(x\! -\! \bar x)\rangle
				\!+\! \int_{0}^{1} \langle \nabla h_{\mu}(G(\bar x)), \left(DG(\bar x\!+\! t(x\! -\! \bar
				x))\! -\! DG(\bar x)\right) (x\! -\! \bar x) \rangle d t
				\\\ge & \langle \nabla g_{\mu}(\bar x), x - \bar x \rangle
				- \|\nabla h_{\mu}(G(\bar x))\|\, \|x - \bar x\| \int_{0}^{1} \|DG(\bar x+t(x
				- \bar x)) - DG(\bar x)\| d t
				\\\ge & \langle \nabla g_{\mu}(\bar x), x - \bar x \rangle
				- \frac{L_{G}M_{\mathcal{B}}}{2}\|x - \bar x\|^2,
		\end{align*}
		where the last inequality holds because $\nabla h_{\mu}(G(\bar x))\in \mathcal{B}$ and $DG$ is Lipschitz continuous.
		Combining the above display with the convexity of $h_{\mu}$ and the definition of $g_{\mu}$ in \eqref{eq-g-mu-def}, we obtain
		\eqref{eq-low-quad-bd}.
		
		Finally, we prove item (iv). Suppose that $G$ is $(-\mathcal{K})$-convex.
		Then for every $u\in \mathcal{K}^{\circ}$, the function $x\mapsto \langle u, G(x)\rangle$ is convex, which implies that
		\begin{equation}\label{eq-k-conv}
			\langle u, G(x) - G(\bar x) \rangle \ge \langle u, DG(\bar x)(x -\bar x)\rangle \quad \forall x, \bar x \in \mathbb{X}.
		\end{equation}
		On the other hand, for every $\mu>0$, by the convexity of $h_{\mu}$, we have, for all $x, \bar x \in \mathbb{X}$,
		\begin{equation*}
			\begin{split}
				h_{\mu}(G(x)) & \ge h_{\mu}(G(\bar x)) + \langle \nabla h_{\mu}(G(\bar x)), G(x) - G(\bar x) \rangle \\
				& \ge h_{\mu}(G(\bar x)) + \langle DG(\bar x)^*\nabla h_{\mu}(G(\bar x)),  x - \bar x\rangle,
			\end{split}
		\end{equation*}
		where the second inequality holds because $\nabla h_{\mu}(G(\bar x))\in \mathcal{K}^{\circ}$ (so that we can set $u = \nabla h_\mu (G(x))$ in \eqref{eq-k-conv}). This proves (iv).
	\end{proof}

	\subsection{Optimality condition}
	We recall the following standard notion of Karush-Kuhn-Tucker (KKT) points for \eqref{opt-conic-dc}. It is well-known that under RCQ, all local minimizers of \eqref{opt-conic-dc} are KKT points; see, for example, the derivation in \cite[Eq.~(2.2)]{21YPL} and \cite[Theorem~6.14]{09RW}.
	\begin{definition}
		\label{def-KKT}
		We say that $x\in G^{-1}(\mathcal{K})$ is a KKT point of
		\eqref{opt-conic-dc}
		if there exists $v \in
		\mathcal{K}^{\circ}$ such that
		$0 \in \partial P_1(x) - \partial P_{2}(x) + \nabla f(x) + DG(x)^{*} v$
				and $\langle v, G(x)\rangle=0$.
	\end{definition}

	Motivated by Definition~\ref{def-KKT} and the fact that $\langle v, G(x)\rangle \le 0$ whenever $x\in G^{-1}({\cal K})$ and $v\in {\cal K}^\circ$, we introduce the next definition as a notion of approximate KKT points.
{\changescolor
	\begin{definition}
	\label{def-epsilon-KKT}
	Let $\epsilon_1\ge 0$, $\epsilon_2\ge 0$ and $\epsilon_3\ge 0$.
	We say that $x\in G^{-1}(\mathcal{K})$ is an $(\epsilon_1, \epsilon_2, \epsilon_3)$-KKT point of
	\eqref{opt-conic-dc}
	if there exist $z\in G^{-1}(\mathcal{K})$ and $v \in
	\mathcal{K}^{\circ}$ such that
		\begin{equation*}
			\begin{split}
				&\mathrm{dist}\left(0, \partial P_1(x) - \partial P_{2}(z) + \nabla f(x) +
				DG(x)^{*} v \right) \le  \epsilon_1,
				\\&   - \langle v, G(x)\rangle \le
				\epsilon_2 \quad \mbox{and}\quad \|x-z\|\le \epsilon_3.
			\end{split}
	\end{equation*}
\end{definition}
}

Notice that the subdifferential $\partial P_2$ is computed at a $z$ close to $x$ (rather than at $x$); we introduce this flexibility because $\partial P_2(\cdot)$ is not continuous in general.
{\changescolor When $\mathbb{X}=\mathbb{R}^{n}$, $\mathbb{Y}=\mathbb{R}^{m}$ and $\mathcal{K}=\mathbb{R}^{m}_{-}$ in \eqref{opt-conic-dc}, 
the notion of an $(\epsilon, \epsilon^2, \epsilon)$-KKT point in the sense of Definition~\ref{def-epsilon-KKT} implies the notion of an $\epsilon$-KKT point as defined in \cite[Definition~6]{25LPB}.
Moreover,  when $\mathbb{X}=\mathbb{R}^{n}$, $\mathbb{Y}=\mathbb{R}^{m}$, $\mathcal{K}=\mathbb{R}^{m}_{-}$ and 
$P_2 = 0$ in \eqref{opt-conic-dc}, a $(\sqrt{\epsilon}, \epsilon, 0)$-KKT point is an $\epsilon$-KKT point in the sense of \cite[Definition~3]{25BDL}. 
}	

	We end this section with the following lemma, which will be useful in our algorithmic analysis later.
	\begin{lemma}\label{lm-mfcq-region}
		Consider \eqref{opt-conic-dc} and \eqref{opt-supp-dc}, and suppose that Assumption~\ref{ass-gen} holds.
		Then there exists a constant $\eta>0$ such that, for every $(x, u) \in
		\Omega_{0} \times \mathcal{B}$, the following conditions cannot be satisfied at the same time:
		\begin{equation*}
			\|DG(x)^{*} u\|^2 \le \eta\quad \mbox{and}\quad -\langle u, G(x)\rangle \le
			\eta.
		\end{equation*}
	\end{lemma}
	\begin{proof}
		Suppose to the contrary that for each $k\in \mathbb{N}_0$ there exists a pair
		$(x^{k}, u^{k})\in \Omega_{0}\times\mathcal{B}$
		such that $\|DG(x^{k})^{*} u^{k}\|^2 \le \frac{1}{k+1}$ and $-\langle u^{k},
		G(x^{k})\rangle \le \frac{1}{k+1}$.
		Since $\mathcal{B}$ is compact, $\{u^{k}\}$ is
		bounded.
		By Assumption~\ref{ass-gen}, $\Omega_{0}$ is bounded, so $\{x^{k}\}$ is bounded too.
		Then there exist convergent subsequences such that
		$x^{k_t} \rightarrow \bar x$ and $u^{k_t} \rightarrow \bar u$ for some $\bar x\in\Omega_{0}$ and $\bar u \in \mathcal{B}$.
		Consequently, it holds that
		\begin{equation}\label{eq-cq}
			G(\bar x) \in \mathcal{K},\, DG(\bar x)^{*} \bar u =
			0,\, -\langle \bar u, G(\bar x)\rangle \le 0.
		\end{equation}		
		On the other hand, since the cone $\mathcal{K}$ has a nonempty interior, by RCQ
		and \cite[Lemma~2.99]{13BS}, there exists $\bar d\in {\mathbb X}$ such
		that
		$
		\bar y \coloneqq  G(\bar x) + DG(\bar x) \bar d \in \mathrm{int}(\mathcal{K})
		$.
		Let $\epsilon>0$ be a sufficiently small constant with $\bar y + \epsilon \bar
		u\in \mathcal{K}$.
		Since $\bar u \in \mathcal{K}^{\circ}$, we have
		\begin{equation*}
			0\ge \langle \bar u, \bar y + \epsilon \bar{u}\rangle
			=  \langle \bar u, G(\bar x)\rangle + \langle DG(\bar x)^{*} \bar u,
			\bar d\rangle  + \epsilon \|\bar{u}\|^2 \ge \epsilon \|\bar{u}\|^2,
		\end{equation*}
		where the last inequality follows from \eqref{eq-cq}.
		Consequently, we get $0 = \bar u\in \mathcal{B}$, which contradicts the fact
		that $\mathcal{B}$ is a base of $\mathcal{K}^{\circ}$.
	\end{proof}
	
	\section{A smoothing moving balls approximation method}
	
	\subsection{Algorithm}\label{sec:alg}
	We present our algorithm, which we call the smoothing moving balls
	approximation method ({\em s}MBA), for solving
	\eqref{opt-supp-dc} (or, equivalently, \eqref{opt-conic-dc}) under Assumptions~\ref{ass-gen} and \ref{ass-sm} in Algorithm~\ref{alg-SMBA} below.
	
	\begin{algorithm}[h]
		\caption{{\em s}MBA: a smoothing moving balls approximation
			method for \eqref{opt-conic-dc} under Assumptions~\ref{ass-gen} and \ref{ass-sm}.}\label{alg-SMBA}
		\begin{algorithmic}[0]
			\setlength{\columnwidth}{\linewidth}
			\STATE {\bf Require.} $\tau_1>0$, $\tau_2>0$, $\hat{L}\ge\check{L}>0$, the
			$x^{0}\in G^{-1}(\mathrm{int}(\mathcal{K}))$ given in Assumption~\ref{ass-gen},\\ \ \ \qquad\qquad and the MSA $\{h_\mu\}_{\mu>0}$ from Assumption~\ref{ass-sm}.
			\STATE {\bf Step 1.} Find $\mu_0>0$ with $g_{\mu_{0}}(x^{0})< 0$.\footnotemark\
			Set the iteration counter $k = 0$.
			\STATE {\bf Step 2.} Choose $\xi^{k}\in \partial P_2 (x^{k})$ and $L^{k, 0}_{f}, L^{k,0}_{g}\in [\check L, \hat L]$.
			Let $i = 0$ and $j = 0$.
			\STATE {\bf Step 3.} Set $L^{k, i}_{f} = 2^{i}L^{k,0}_{f}$, $L^{k,
				j}_{g} = 2^{j}L^{k,0}_{g}$. Solve the following subproblem for the
				\\ \qquad\qquad unique solution $x^{k,i,j}$ and a Lagrange multiplier $\lambda_{k,i,j}\ge 0$:
			\begin{equation}\label{opt-sub}
				\begin{array}{rl}
					\min\limits_{x\in{\mathbb X}} &
					\psi_{k,i}(x) \coloneqq P_1(x) + \langle \nabla f(x^{k}) - \xi^{k},\, x-x^{k} \rangle
					+ \frac{L^{k,i}_{f}}{2}\|x-x^{k}\|^2
					\\ \mbox{s.t.}&
 g_{k,j}(x) \coloneqq  g_{\mu_{k}}(x^{k}) + \langle \nabla
					g_{\mu_{k}}(x^{k}),\, x - x^{k} \rangle  + \frac{L^{k,j}_{g}}{2\mu_{k}}\|x-x^{k}\|^2 \le 0.
				\end{array}
			\end{equation}
			\STATE {\bf Step 4.} If $g_{\mu_{k}}({x}^{k,i,j}) \le 0$ and
			\begin{equation}\label{eq-des-suf}
				\psi(x^{k, i, j}) \le \psi(x^{k}) - \frac{\tau_1\mu_{k}  + \tau_2 \lambda_{k, i, j}
				}{2\mu_{k}}\|x^{k,i,j} - x^{k}\|^2
			\end{equation}
			
			\qquad\qquad hold, let $i_k = i$, $j_k = j$ and go to {\bf Step 5}; else if
			$g_{\mu_{k}}({x}^{k,i,j}) > 0$, let \\ \qquad\qquad $j\leftarrow j+1$ and go to {\bf Step 3};
			else let $i \leftarrow  i + 1$ and $j\leftarrow j+1$, and\\ \qquad\qquad go to {\bf Step 3}.
			\STATE {\bf Step 5.} Let $x^{k+1} = x^{k,i_k,j_k}$, $\lambda_{k+1} =
			\lambda_{k,i_k,j_k}$, $L^{k}_{f} = L^{k,i_k}_{f}$ and $L^{k}_{g} =
			L^{k,j_k}_{g}$. Choose\\ \qquad\qquad $\mu_{k+1}\in (0,\mu_k)$. Update $k\leftarrow k+1$ and go to {\bf Step 2}.
		\end{algorithmic}
	\end{algorithm}
\footnotetext{Recall that $g_\mu$ is defined in \eqref{eq-g-mu-def}.}

In essence, the algorithm is an adaptation of the moving balls approximation method from \cite{10AST} (see also \cite{16ST} for a specialization to handle a single {\em smooth} inequality constraint, and \cite{21YPL} for the consideration of a difference-of-convex objective and the incorporation of linesearch techniques) to solving \eqref{opt-supp-dc}. Since $g_{\cal B}$ in \eqref{opt-supp-dc} is nonsmooth in general, we replace $g_{\cal B}$ by the smooth function $g_{\mu_k}$ from \eqref{eq-g-mu-def} in the $k$th iteration and construct the corresponding moving balls approximation subproblem as in \eqref{opt-sub}. Compared with the recent works \cite{25BDL,25LPB} which studied \eqref{alg-SMBA} with ${\cal K} = \R^m_-$ and whose algorithms involve subproblems with possibly multiple inequality constraints (and require an iterative solver), our resulting subproblem \eqref{opt-sub} always involves a {\em single} inequality constraint, and can be solved efficiently via some simple root-finding procedures if the proximal mapping of $\gamma P_1$ can be computed efficiently for all $\gamma > 0$; see, e.g., \cite[Appendix~A]{21YPL} for the case when $P_1$ is the $\ell_1$ norm {\changescolor where such subproblems can be solved efficiently and {\em exactly}}.

Our algorithm starts with the $x^{0}\in G^{-1}(\mathrm{int}(\mathcal{K}))$ given in Assumption~\ref{ass-gen}. It proceeds by identifying a $\mu_0$ satisfying $g_{\mu_0}(x^0) < 0$, whose existence we argue briefly as follows. Indeed, we have $\sigma_{\mathcal{B}}(G(x^{0}))<0$ in view of Lemma~\ref{lm-base}. Hence, there exists a sufficiently small $\mu_{0}>0$ such that
		\begin{equation}\label{hahahaha}
		g_{\mu_{0}}(x^{0}) \le \sigma_{\mathcal{B}}(G(x^{0})) +
		\alpha_{3} \mu_{0}<0,
		\end{equation}
		where we use \eqref{eq-gmu-appr} to get the
		first inequality.
		This shows the existence of $\mu_{0}$ in {\bf Step 1} of Algorithm~\ref{alg-SMBA}.

The crux of the algorithm involves two ingredients: the subproblem \eqref{opt-sub} in {\bf Step 3} and the linesearch procedure in {\bf Step 4} for feasibility and the descent condition \eqref{eq-des-suf}. We will establish the well-definedness of our algorithm by arguing the well-definedness of the subproblem \eqref{opt-sub} and the finite termination of the linesearch procedure in Lemma~\ref{lm-well-defined} below. We start with an auxiliary lemma on the Lipschitz continuity modulus of the gradient of $g_\mu$ in \eqref{eq-g-mu-def}.

	\begin{lemma}[Lipschitz continuity of $\nabla g_\mu$]\label{lm-g-mu-grad-Lip}
		Consider \eqref{opt-conic-dc} and \eqref{opt-supp-dc}, and suppose that Assumptions~\ref{ass-gen} and \ref{ass-sm} hold.
		Let $g_{\mathcal{B}}$, $M_{\mathcal{B}}$, $g_{\mu}$ and $M_{G}$ be defined in \eqref{opt-supp-dc}, \eqref{eq-M-B-def}, \eqref{eq-g-mu-def} and \eqref{eq-M-G-def}, respectively.
		Let $\mu_{0}$ and $\check{L}$ be given in Algorithm~\ref{alg-SMBA}.\footnote{According to the discussion leading to \eqref{hahahaha}, we know that $\mu_0$ is well-defined.} Define
		\begin{align}
			& M_{R} {\changescolor\coloneqq} \sqrt{\mu_{0}^2\check{L}^{-2}M_G^2 M_{\mathcal{B}}^2 -
				2 \mu_{0}\check{L}^{-1} g^{*}_{\mathcal B}}+\mu_{0} \check{L}^{-1} M_{G} M_{\mathcal{B}}, \label{eq-def-MR}
				\\ &\hat{\Omega}_{0} {\changescolor\coloneqq}  \Omega_{0} + B\left(0, M_{R}\right), \label{eq-def-hat Omega}
		\end{align}
		where $g^{*}_{\mathcal B} {\changescolor\, \coloneqq} \inf_{x\in \Omega_{0}}g_{\mathcal{B}}(x)$.
		Then, for every $\mu>0$, $g_{\mu}$ has an $L_{\mu}$-Lipschitz continuous gradient on $\mathrm{conv}(\hat \Omega_{0})$, where
		\begin{align}
			&L_{\mu} {\changescolor\coloneqq} M_{\mathcal{B}}L_G + \alpha_{1}\hat M_{G}^2 + \alpha_{2} \hat
			M_{G}^2/\mu \label{eq-L-mu}
		\end{align}
		with $\hat M_{G} \coloneqq  \sup \{\| DG(x)\| : x \in \mathrm{conv}(\hat \Omega_{0})\} < \infty.$
	\end{lemma}
	\begin{proof}
		From Assumption~\ref{ass-gen} and {\changescolor Lemma~\ref{lm-base}, we} see that $g_{\mathcal{B}}^* \in (-\infty, 0)$.
		Hence, $M_{R} \in (0, \infty)$ and $\hat \Omega_{0}$ is compact.
		For all $\mu >0$ and every $x,z\in \mathrm{conv}(\hat \Omega_{0})$, we
		have
		\begin{align*}
				& \|\nabla g_{\mu}(z) - \nabla g_{\mu}(x)\| = \| DG(z)^{*}\nabla
				h_{\mu}(G(z)) - DG(x)^{*}\nabla h_{\mu}(G(x))\|
				\\ \le & \| \left(DG(z)- DG(x)\right)^{*}\nabla h_{\mu}(G(z))\|
				+ \| DG(x)^{*}\left(\nabla h_{\mu}(G(z)) - \nabla
				h_{\mu}(G(x))\right)\|
				\\ \le & \|\nabla h_{\mu}(G(z))\|\, \| DG(z)- DG(x)\|
				+ \| DG(x)\|\, \|\nabla h_{\mu}(G(z)) - \nabla h_{\mu}(G(x))\|
				\\ \le & L_G M_{\mathcal{B}} \| z - x \| + (\alpha_{1} + \alpha_{2}/\mu) \hat
				M_{G} \| G(z) - G(x)\|
				\\ \le & \left(M_{\mathcal{B}}L_G + \alpha_{1}\hat M_{G}^2  + \alpha_{2} \hat
				M_{G}^2/\mu\right) \| z - x \|,
		\end{align*}
		where the third inequality follows from the Lipschitz continuity of $DG$ and $\nabla
		h_{\mu}$.
	\end{proof}

	\begin{lemma}[Well-definedness of Algorithm~\ref{alg-SMBA}]\label{lm-well-defined}
		Consider \eqref{opt-conic-dc} and \eqref{opt-supp-dc}, and suppose that Assumptions~\ref{ass-gen} and \ref{ass-sm} hold.
		Then Algorithm~\ref{alg-SMBA} is well-defined in the sense that
\begin{itemize}
  \item the $\mu_0$ in {\bf Step 1} can be found;
  \item the subproblem \eqref{opt-sub} is well-defined;
  \item there exists $\iota$ (independent of $k$) such that {\bf Step 4} is invoked at most $\iota$ times every iteration.
\end{itemize}
		More precisely, with $g_{\mu}$ and $L_{\mu}$ defined in \eqref{eq-g-mu-def} and \eqref{eq-L-mu} respectively, the following statements hold for each iteration $k$.
		\begin{enumerate}[{\rm (i)}]
			\item It holds that $x^{k}\in \Omega_{0}$ and is strictly feasible for \eqref{opt-conic-dc} with $g_{\mu_{k}}(x^{k})<0$.
			\item For each $i,j\in \mathbb{N}_{0}$, \eqref{opt-sub} has a unique
			solution $x^{k,i,j}$ and there exists a
			$\lambda_{k, i, j} \in \mathbb{R}_{+}$ such that
			\begin{align}
				& 0 \in \partial P_1(x^{k,i,j})\!  + \! \nabla f(x^{k})\! -\! \xi^{k}
				\!+\! \lambda_{k,i,j} \nabla g_{\mu_{k}}(x^{k}) \!+\! \frac{\changescolor\widetilde L_{k,i,j}}{\mu_{k}} (x^{k,i,j} \!-\! x^{k}), \label{eq-res}\\
				& \lambda_{k,i,j} g_{k,j}(x^{k,i,j}) = 0 \quad \mbox{and}\quad
				g_{k,j}(x^{k,i,j})\le 0, \label{eq-slack}
			\end{align}
			where ${\changescolor\widetilde L_{k,i,j} \coloneqq} \mu_{k} L^{k,i}_{f} + \lambda_{k,i,j} L^{k,j}_{g}$.
			Moreover, it holds that, for all $x\in {\mathbb X}$,
			\begin{align}
				P_1(x^{k,i,j}) & \le P_1(x) +\langle \nabla f(x^{k}) - \xi^{k},\, x-x^{k,i,j} \rangle + \lambda_{k,i,j}g_{k,j}(x)
				\nonumber\\ & \quad + \frac{L^{k,i}_{f}}{2}\|x-x^{k}\|^2
				-  \frac{\lambda_{k,i,j} L_{g}^{k,j}}{2\mu_{k}} \|x - x^{k,i,j}\|^2 \label{eq-lambda-bd-hat x}
\end{align}
and
\begin{align}
				\psi(x^{k,i,j})
				& \le  f(x^{k}) + P_1(x) - P_{2}(x^{k}) + \langle \nabla f(x^{k}) - \xi^{k}, x - x^{k}\rangle
				\nonumber\\ &\quad + \lambda_{k,i,j} g_{k,j}(x)
				+ \frac{L^{k,i}_{f}}{2}\|x-x^{k}\|^2
				\nonumber\\ &\quad + \frac{L_{f} - L^{k,i}_{f}}{2}\|x^{k,i,j} - x^{k}\|^2
				- \frac{\changescolor\widetilde L_{k,i,j}}{2\mu_{k}}
				\|x - x^{k,i,j}\|^2. \label{eq-ball-prox}
			\end{align}
			
			\item The inner loop stops for some
			integers $i_{k}\le j_k \le \bar i + \bar j$, where
			\begin{align}
				&\bar i {\changescolor\, \coloneqq}  \max\{0,\, \lceil \log_{2}\left((L_f+\tau_1)/\check{L}\right) \rceil\},
				\label{eq-def-bar s}\\
				&\bar j {\changescolor\, \coloneqq}  \max\{0,\, \lceil \log_{2}\left(\tau_2/\check{L}\right)  \rceil, \lceil \log_{2}\left(\mu_{0}L_{\mu_0}/\check{L}\right)
				\rceil\}. \label{eq-def-bar j}
			\end{align}
			Moreover, it holds that
			\begin{align}
\check{L}\le L^k_f \le M_{L}\ \ {and}\ \ \check{L}\le L^k_g \le M_{L}, \label{eq-bd-MLf}
\end{align}
where
\begin{align}
M_{L} {\changescolor\, \coloneqq} \hat{L}\max\{2, 2(L_{f} + \tau_1)/\check{L}\}\max\{2, 2\tau_2/\check{L}, 2\mu_{0}L_{\mu_{0}}/\check{L}\}.\label{eq-bd-MLg}
			\end{align}
			\item It holds that
			\begin{equation}\label{eq-f-des}
				\psi(x^{k+1}) \le \psi(x^{k}) - \frac{\tau_1 \mu_{k} + \tau_2\lambda_{k+1}}{2\mu_{k}} \|x^{k+1} - x^{k}\|^2\ \ {and}\ \
g_{\mu_k}(x^{k+1})\le 0.
			\end{equation}
		\end{enumerate}
	\end{lemma}
	
	\begin{proof}
		We prove the lemma by induction. First, according to the discussion leading to \eqref{hahahaha}, we see that $\mu_0$ is well-defined and $g_{\mu_0}(x^0) < 0$.
		We thus assume, as an induction hypothesis, that, for some $k\in\mathbb{N}_{0}$, $x^{k} \in \Omega_{0}$
		and is strictly feasible for \eqref{opt-conic-dc} with
		$g_{\mu_{k}}(x^{k})<0$.
		We will first show that properties (ii), (iii) and (iv) hold at the $k$th iteration.
		Then we will prove that $x^{k+1}\in \Omega_{0}$ and is strictly feasible for problem
		\eqref{opt-conic-dc} with $g_{\mu_{k+1}}(x^{k+1})<0$.
		This will complete the induction step and establish property (i).
		Since property (i) implies (ii), (iii) and (iv) hold,
		the proof is complete.
		
		We first prove property (ii) under the induction hypothesis. Indeed, at the $k$th iteration and for any $i,j\in\mathbb{N}_{0}$, since
		$g_{\mu_{k}}(x^{k})<0$, $\mu_{k}>0$ and $L^{k,j}_{g}>0$, it holds that
		$- 2\mu_{k}^{-1}(L^{k,j}_{g})^{-1} g_{\mu_{k}}(x^{k})>0$. Notice that the $g_{k,j}$ in \eqref{opt-sub} can be rewritten as
		$g_{k,j}(x) = \frac{L^{k,j}_{g}}{2\mu_{k}}\left(\|x-\hat x^{k,j}\|^2 - {R}_{k,j}^2\right)$
		with
		\begin{align}
			& \hat x^{k,j} = x^{k} - \mu_{k}(L^{k,j}_{g})^{-1} \nabla g_{\mu_{k}}(x^{k}) \label{eq-def-center22},\\
			&R_{k,j} = \mu_{k} (L^{k,j}_{g})^{-1} \sqrt{\| \nabla g_{\mu_{k}}(x^{k})\|^2  -
				2\mu_{k}^{-1}L_{g}^{k,j} g_{\mu_{k}}(x^{k})}, \notag
		\end{align}
		we see that $R_{k,j}$ is (well-defined and) positive.
		Thus, \eqref{opt-sub} is to minimize a strongly convex
		function over the closed ball $B(\hat{x}^{k,j}, R_{k,j})$ with positive radius.
		In view of \cite[Theorem 2.9.2]{02Zalinescu}, \eqref{opt-sub} has a
		unique solution $x^{k,i,j}\in{\mathbb X}$ with
		a Lagrange multiplier $\lambda_{k, i, j} \in \mathbb{R}_{+}$ such that
		\eqref{eq-slack} holds and
		$
			x^{k,i,j} = \mathrm{argmin}_{x\in{\mathbb X}} \mathcal{L}_{k,i,j}(x,
			\lambda_{k,i,j})$,
		where $\mathcal{L}_{k,i,j}(x, \lambda) := \psi_{k,i}(x)  +
			\lambda g_{k,j}(x)$ for $x\in \mathbb{X}$ and $\lambda\in \mathbb{R}$.
		The relation \eqref{eq-res} follows from $0\in \partial_{x}\mathcal{L}_{k,i,j}(x^{k,i,j},
		\lambda_{k,i,j})$.
		Next, note that  $\mathcal{L}_{k,i,j}(\cdot, \lambda_{k,i,j})$ is
		${\changescolor\widetilde L_{k,i,j}}/\mu_k$-strongly convex. This together with \eqref{eq-slack} gives, for all $x\in \mathbb{X}$,
		\begin{align*}
				& P_1(x^{k,i,j}) + \langle \nabla f(x^{k}) - \xi^{k},\, x^{k,i,j}-x^{k}
				\rangle + \frac{L^{k,i}_{f}}{2}\|x^{k,i,j} - x^{k}\|^2
				\\ & = {\changescolor \mathcal{L}_{k,i,j}}(x^{k,i,j}, \lambda_{k,i,j})
				\le {\changescolor \mathcal{L}_{k,i,j}}(x, \lambda_{k,i,j}) - \frac{\changescolor\widetilde L_{k,i,j}}{2\mu_{k}}
				\|x - x^{k,i,j}\|^2
				\\ & {\changescolor =} P_1(x) \!+\! \langle \nabla f(x^{k}) - \xi^{k},\, x-x^{k} \rangle
				\!+\! \frac{L^{k,i}_{f}}{2}\|x-x^{k}\|^2 \!+\! \lambda_{k,i,j}g_{k,j}(x) \!-\! \frac{\changescolor\widetilde L_{k,i,j}}{2\mu_{k}}
				\|x - x^{k,i,j}\|^2.
		\end{align*}
		Rearranging terms in the above display, one can see that
		\eqref{eq-lambda-bd-hat x} holds at the $k$th iteration.
		On the other hand, by the convexity of $P_{2}$, we have
		\begin{equation*}
			-P_{2}(x^{k,i,j}) + \langle \xi^{k}, x^{k,i,j} - x^{k} \rangle \le -P_{2}(x^{k}).
		\end{equation*}
		In addition, by the descent lemma, it holds that
		\begin{equation*}
			f(x^{k,i,j}) \le f(x^{k}) + \langle \nabla f(x^{k}),\, x^{k,i,j} - x^{k}\rangle
			+ \frac{L_{f}}{2}\|x^{k,i,j} - x^{k}\|^2.
		\end{equation*}
		Summing the above three displays and rearranging terms, one can see that
		\eqref{eq-ball-prox} holds at the $k$th iteration.
		
		Next, we show that properties (iii) and (iv) hold at the $k$th iteration.
		Let $i,j\in\mathbb{N}_{0}$.
		Let $\hat{x}^{k,j}$ be defined in \eqref{eq-def-center22}, and $g^{*}_{\mathcal B}$, $M_{R}$ and $\hat \Omega_{0}$ be defined in Lemma~\ref{lm-g-mu-grad-Lip}.
		Then
		\begin{equation*}
			\begin{split}
				& \|x^{k,i,j} - x^{k}\|  \le \|x^{k,i,j} - \hat{x}^{k,j}\| + \|\hat{x}^{k,j} -
				x^{k}\|
				\\ & \le \sqrt{\mu_{k}^2 (L^{k,j}_{g})^{-2}\| \nabla g_{\mu_{k}}(x^{k})\|^2  -
					2\mu_{k}(L^{k,j}_{g})^{-1} g_{\mu_{k}}(x^{k})} + \mu_{k} (L^{k,j}_{g})^{-1} \|
				\nabla g_{\mu_{k}} (x^{k})\|
				\\ & \le \sqrt{\mu_{0}^2\check{L}^{-2}M_G^2 M_{\mathcal{B}}^2 - 2
					\mu_{0}\check{L}^{-1} g^{*}_{\mathcal B}} + \mu_{0} \check{L}^{-1} M_{G} M_{\mathcal{B}} =
				M_{R},
			\end{split}
		\end{equation*}
		where the second inequality follows from $\|x^{k,i,j} - \hat{x}^{k,j}\|\le R_{k,j}$ since $x^{k,i,j}$ is feasible for (\ref{opt-sub}),
		and the last inequality holds because of \eqref{eq-bd-nabla g_mu}, $0<\mu_{k}\le \mu_{0}$, $0<\check{L}\le
		L^{k,j}_{g}$ and $-\infty < g_{\mathcal{B}}^{*}\le g_{\cal B}(x^k)\le g_{\mu_{k}}(x^{k}) < 0$ (see \eqref{eq-gmu-appr} for $g_{\cal B}\le g_{\mu_k}$).
		Recalling that $x^{k}\in \Omega_{0}$, we have $x^{k,i,j}\in \hat \Omega_{0}=\Omega_{0} +
		B\left(0, M_R\right)$.
		From Lemma~\ref{lm-g-mu-grad-Lip}, $g_{\mu_{k}}$ is
		$L_{\mu_{k}}$-smooth on $\mathrm{conv}(\hat \Omega_{0})$.
		By the descent lemma, we have
		\begin{equation}\label{eq-g-ls}
			\begin{split}
				g_{\mu_{k}}(x^{k,i,j}) & \le g_{\mu_{k}}(x^{k}) + \langle \nabla
				g_{\mu_{k}}(x^{k}),\, x^{k,i,j} - x^k \rangle
				+ \frac{L_{\mu_{k}}}{2}\|x^{k,i,j} - x^{k}\|^2
				\\& \le - \frac{L_{g}^{k,j} - \mu_{k} L_{\mu_{k}}}{2\mu_{k}}\|x^{k,i,j} -
				x^{k}\|^2,
			\end{split}
		\end{equation}
		where we use the feasibility of $x^{k,i,j}$ for \eqref{opt-sub} to
		get the last inequality.
		On the other hand, by setting $x = x^{k}$ in \eqref{eq-ball-prox}, we have
		\begin{equation}\label{eq-psi-ls}
			\begin{split}
				\psi(x^{k,i,j})
				& \le \psi(x^{k}) \!+\! \lambda_{k,i,j} g_{\mu_{k}}(x^{k}) \!-\!  \frac{1}{2}\left( 2
				L^{k,i}_{f} \!-\! L_f \!+\! \frac{\lambda_{k,i,j} L^{k,j}_{g}}{\mu_{k}}
				\right)\|x^{k,i,j} \!-\! x^k\|^2
				\\ & \le \psi(x^{k}) -  \frac{1}{2}\left( 2 L^{k,i}_{f} - L_f +
				\frac{\lambda_{k,i,j} L^{k,j}_{g}}{\mu_{k}}  \right) \|x^{k,i,j} - x^k\|^2,
			\end{split}
		\end{equation}
		where the last inequality follows since $\lambda_{k,i,j}\ge 0$ and
		$g_{\mu_{k}}(x^{k})< 0$.

		Let $\bar i$ and $\bar j$ be defined in \eqref{eq-def-bar s} and \eqref{eq-def-bar j}, respectively.
		Then, it holds that
		\begin{align}
			& \max\{\check{L}, L_{f} + \tau_1\} \le 2^{\bar i}\check{L} \le  \max\{2\check{L}, 2(L_{f} + \tau_1)\}, \label{eq-sL-bd}\\
			& \max\{\check{L}, \tau_2, \mu_{0}L_{\mu_{0}} \} \le 2^{\bar{j}} \check{L} \le \max\{2\check{L}, 2\tau_2, 2\mu_{0}L_{\mu_{0}} \}. \label{eq-tL-bd}
		\end{align}
		From this, one can see that, if $i\ge \bar i$ and $j\ge \bar j$, then
		\begin{align*}
			& 2L_{f}^{k,i} = 2^{i+1}L^{k,0}_{f} \ge 2^{\bar i} \check{L} \ge L_f + \tau_1, \\
			& L^{k, j}_{g} = 2^{j}L_{g}^{k,0} \ge 2^{\bar{j}} \check{L} \ge \max\{\tau_2, \mu_{0}
			L_{\mu_{0}}\} \ge  \max\{\tau_2, \mu_{k} L_{\mu_{k}}\},
		\end{align*}
		where the last inequality follows from the definition of $L_{\mu_{k}}$ in \eqref{eq-L-mu} and the fact that $\{\mu_{k}\}$ is decreasing.
		Combining these with \eqref{eq-g-ls} and \eqref{eq-psi-ls}, we have
		\eqref{eq-des-suf} and $g_{\mu_{k}}(x^{k,i,j}) \le 0$ hold whenever
		$i\ge \bar i$ and $j\ge \bar j$.
		Hence, from the update rule, at the $k$th iteration, we see that the inner loop must terminate for some
		$i_k \le j_k \le \bar i + \bar j$, and output an $(x^{k+1},
		\lambda_{k+1})$ satisfying \eqref{eq-f-des}.
		Moreover, from \eqref{eq-sL-bd} and \eqref{eq-tL-bd}, we have
		\begin{align*}
			\check{L} \le L^{k,0}_{f} \le L_{f}^{k} \!=\! 2^{i_{k}} L^{k,0}_{f}
			\le 2^{\bar i+\bar j} \hat{L} \le M_{L}, \\
\check{L} \le L^{k,0}_{g} \le L_{g}^{k} \!=\! 2^{j_{k}} L^{k,0}_{g}
			\le 2^{\bar i + \bar j} \hat L
			\le M_{L},
		\end{align*}
where $M_{L}$ is given in \eqref{eq-bd-MLg}.
		This shows that \eqref{eq-bd-MLf} holds at the $k$th
		iteration.
		Consequently, properties (iii) and (iv) also hold at the $k$th iteration.
		
		Finally, from \eqref{eq-gmu-str-des}, we
		obtain
		$
			g_{\mu_{k+1}}(x^{k+1}) \le g_{\mu_{k}}(x^{k+1}) - \alpha_{4} (\mu_{k} -
			\mu_{k+1}) <0
        $,
		where the last inequality holds because $g_{\mu_{k}}(x^{k+1}) \le 0$ (see \eqref{eq-f-des}), $\mu_{k+1}<\mu_{k}$ and $\alpha_{4}>0$.
		Combining this with \eqref{eq-gmu-appr}, we deduce that
		$
		\sigma_{\mathcal{B}}(G(x^{k+1})) \le g_{\mu_{k+1}}(x^{k+1})<0
		$.
		Applying Lemma~\ref{lm-base}, we see that $G(x^{k+1})\in
		\mathrm{int}(\mathcal{K})$.
		From this and \eqref{eq-f-des}, we conclude that $x^{k+1} \in \Omega_{0}$
		and is strictly feasible for \eqref{opt-conic-dc}.
		This completes the proof.
	\end{proof}
	\begin{remark}
		One can see that the $g_{k,j}$ in \eqref{opt-sub} can be reformulated as:
		\begin{equation}\label{eq-sub-const-ball}
			g_{k,j}(x) = \frac{L^{k,j}_{g}}{2\mu_{k}}\left(\|x-\hat x^{k,j}\|^2 - {R}_{k,j}^2\right),
		\end{equation}
		where
		\begin{align}
			& \hat x^{k,j} = x^{k} - \mu_{k}(L^{k,j}_{g})^{-1} \nabla g_{\mu_{k}}(x^{k}) \label{eq-def-center},\\
			&R_{k,j} = \mu_{k} (L^{k,j}_{g})^{-1} \sqrt{\| \nabla g_{\mu_{k}}(x^{k})\|^2  -
				2\mu_{k}^{-1}L_{g}^{k,j} g_{\mu_{k}}(x^{k})}. \label{eq-def-radius}
		\end{align}
		Thus, the feasible set of \eqref{opt-sub} is $B(\hat{x}^{k,j}, {R}_{k,j})$, and $R_{k,j} > 0$ thanks to Lemma~\ref{lm-well-defined}(i).
	\end{remark}
    
	\subsection{Complexity analysis}\label{sec:complexity} In this subsection, we derive the iteration complexity of Algorithm~\ref{alg-SMBA} for finding an {\changescolor $(\epsilon_1, \epsilon_2, \epsilon_1 \sqrt{\epsilon_2})$}-KKT point of \eqref{opt-conic-dc} in the sense of Definition~\ref{def-epsilon-KKT}. Notice that the $k$th iteration of the algorithm is essentially solving the subproblem corresponding to applying the moving balls approximation method to \eqref{opt-supp-dc} with $g_{\mu_k}$ in place of $g_{\cal B}$. It is thus natural to expect that one has to carefully control how the $\mu_k$ (and hence the $g_{\mu_k}$) changes from iteration to iteration to guarantee the convergence of $\{x^k\}$. Here, we specifically consider the following assumption on $\{\mu_k\}$.\footnote{Note that the requirements on positivity and monotonicity in Assumption~\ref{ass-muk} were explicitly used in Algorithm~\ref{alg-SMBA} as the (basic) criterion for choosing $\{\mu_k\}$; see {\bf Step 5}. Here, we explicitly include this basic criterion as part of Assumption~\ref{ass-muk} for easy reference.}
	{\changescolor
		\begin{assumption}\label{ass-muk}
			The sequence $\{\mu_{k}\}$ is positive, decreasing (i.e., $\mu_{k+1} < \mu_k$ for all $k\in \mathbb{N}_0$), and
			satisfies 
			\begin{equation}\label{eq-mu-infty}
				\lim_{k\to\infty}\mu_k = 0 
				\quad \mbox{and} \quad 
				S_{K} \coloneqq \sum_{k=\lceil K/2 \rceil}^{K} \mu_{k} \to \infty
				\quad \mbox{as} \quad K \to \infty.
			\end{equation}
	\end{assumption}}	
	
	{\changescolor We next present a specific choice of $\{\mu_{k}\}$ that satisfies Assumption~\ref{ass-muk}.}
	{\changescolor
	\begin{proposition}\label{prop-muk}
		Let $\{r_{k}\}\subset (0, 1)$
		be a non-decreasing sequence such that ${\bar r} \coloneqq \sup_{k\in \mathbb{N}_{0}} r_{k} < 1$.
		Let $n_0 \in \mathbb{N}_{0}$ and $\nu_{0} \in (0, 1]$.
		For each $k\in \mathbb{N}_{0}$, let $k_1$, $k_2\in
		\mathbb{N}_{0}$ satisfy $k = k_2 (n_0+1) + k_1$ and $k_1\le n_0$.
		Let $\mu_0 > 0$ and define\!\!
		\begin{equation}\label{eq-def-muk}
			\mu_{k} \coloneqq  \mu_{k_1,k_2} \coloneqq  \mu_{0}\left({k_2 (n_0 +1) + \nu_{0} k_1 +
				1}\right)^{-r_k}\quad \forall\, k\in \mathbb{N}.
		\end{equation}
		Then $\{\mu_k\}$ satisfies Assumption~\ref{ass-muk}.
		Moreover, we have
		\begin{equation}\label{eq-sum-K-lb}
			\sum_{k = \lceil K/2 \rceil}^{K} \mu_{k} \ge \frac{\mu_{0}}{2^{2{\bar r}+1}} K^{1-{\bar r}} \quad \forall K\in \mathbb{N}_{0}.
		\end{equation}
\end{proposition}}
\begin{remark}
In \eqref{eq-def-muk},
  by choosing $\{r_k\}$ to be a constant sequence and $\nu_0$ to be small relative to $n_0$ (say, $n_0 = 300$ and $\nu_0 = 10^{-4}$), the corresponding $\mu_k$ can change very slowly for $n_0$ iterations (essentially constant) and experience an abrupt change only after every $n_0$ iterations. This choice allows us to mimic the usual smoothing strategy where subproblems with a fixed $\mu_k$ will be solved repeatedly before updating $\mu_k$, as well as the multiple epoch scheme considered in \cite[Section~4.2]{21BW}.
\end{remark}
\begin{proof}[Proof of Proposition~\ref{prop-muk}]
{\changescolor
		Let $k\in \mathbb{N}_{0}$.
		Clearly, there exists a unique pair $k_1,k_2\in \mathbb{N}_{0}$ such that $k = k_2 (n_0+1) + k_1$ and $k_1\le n_0$.
		Since $0<\nu_0 \le 1$, we have from \eqref{eq-def-muk} that, for every $k\in \mathbb{N}_{0}$
		\begin{align*}
			\mu_{k} \le \mu_{0}\nu_{0}^{-r_k} \left({k_2 (n_0 +1) +  k_1 +1}\right)^{-r_k}
			\le \mu_{0}\nu_{0}^{-\bar{r}} (k+1)^{-r_0},
		\end{align*}
		which shows $\lim_{k\to \infty} \mu_k = 0$.
		Moreover, it holds that, for every $k\in \mathbb{N}_{0}$
		\begin{equation*}
			\mu_{k} \ge \mu_{0}\left({k_2 (n_0 +1) +  k_1 +1}\right)^{-r_k} = \mu_0 (k+1)^{- r_k} \ge \mu_0 (k+1)^{- {\bar r}}.
		\end{equation*}	
		For any integer $K\in \mathbb{N}_0$, using the above display, we have	
		\begin{equation*}
			\begin{split}
				& S_{K} = \sum_{k = \lceil K/2 \rceil}^{K} \mu_{k} 
				 \ge \mu_{0} \sum_{k = \lceil K/2 \rceil}^{K} (k+1)^{-{\bar r}}
                 \ge \mu_0 \sum_{k = \lceil K/2 \rceil}^{K}\int_{k}^{k+1}(t + 1)^{-{\bar r}} dt
				\\& = \mu_{0} \int_{\lceil K/2 \rceil}^{K+1} (t + 1)^{-{\bar r}} dt
				\overset{\mathrm{(a)}}{\ge}  \mu_{0} \frac{K + 1 - \lceil K/2 \rceil}{(K+2)^{\bar r}} 
				 \overset{\mathrm{(b)}}{\ge} \frac{\mu_{0}}{2} K (K + 2)^{-{\bar r}} 
				\ge \frac{\mu_{0}}{2^{2\bar r+1}} K^{1-{\bar r}},
			\end{split}
		\end{equation*}
		where (a) holds because $(t+1)^{-{\bar r}} \ge (K+2)^{-{\bar r}}$ for all $t \in [\lceil K/2\rceil, K+1]$ , and (b) holds because $\lceil K/2 \rceil \le K/2 + 1$.
		This shows \eqref{eq-sum-K-lb} and $\lim_{K\to\infty} \sum_{k=[K/2]}^{K} \mu_{k} = \infty$.}
		
		Next, we claim that $\{\mu_k\}$ is decreasing.
		Indeed, for any $k\in\mathbb{N}_0$, if $k_1 < n_0$, then $k+1 = k_2 (n_0+1) + k_1+1$ with $k_1+1\le n_0$.
		That is $\mu_{k+1} = \mu_{k_1+1,k_2}$.
		It follows that
		\begin{equation*}
			\begin{split}
				\mu_{k} - & \mu_{k+1}  = \mu_{0}\left({k_2 (n_0 +1) + \nu_{0} k_1 + 1}\right)^{-r_k}
				- \mu_{0}\left({k_2 (n_0 +1) + \nu_{0} (k_1 + 1) + 1}\right)^{-r_{k+1}}
				\\ & \ge \mu_{0}\left({k_2 (n_0 +1) + \nu_{0} k_1 + 1}\right)^{-r_{k+1}}
				- \mu_{0}\left({k_2 (n_0 +1) + \nu_{0} (k_1 + 1) + 1}\right)^{-r_{k+1}} > 0,
			\end{split}
		\end{equation*}
		where {\changescolor the first inequality} holds because of $0 < r_{k} \le r_{k+1}$, {\changescolor and the positivity holds because $\nu_0 > 0$.}
		
		On the other hand, if $k_1 = n_0$, then $k+1 = (k_2 + 1) (n_0 + 1)$.
		That is $\mu_{k+1} = \mu_{0, k_2 + 1}$.
		It holds that
		\begin{equation*}
			\begin{split}
				\mu_{k} - \mu_{k+1} & = \mu_{0}\left({k_2 (n_0 +1) + \nu_{0} n_0 + 1}\right)^{-r_k}
				- \mu_{0}\left({(k_2 + 1) (n_0 +1) + 1}\right)^{-r_{k+1}}
				\\ & \ge \mu_{0}\left({k_2 (n_0 +1) + \nu_{0} n_0 + 1}\right)^{-r_{k+1}}
				- \mu_{0}\left({(k_2 + 1) (n_0 +1) + 1}\right)^{-r_{k+1}}>0,
			\end{split}
		\end{equation*}
		where {\changescolor the first inequality} holds thanks to $0 < r_{k} \le r_{k+1}$, and the positivity follows from the fact that $\nu_0\in (0,1]$.
\end{proof}
	
	We need the following constants in the subsequent analysis:
	\begin{align}
		&\textstyle M_f \coloneqq \sup_{x\in \Omega_{0}} \|\nabla f(x)\|, \label{eq-def-Mf}\\
		&\textstyle M_{P_1}\coloneqq \sup_{x\in \hat\Omega_{0}}\left\{ \|\zeta\| : \zeta\in \partial
		P_1(x)\right\}
		\ \mbox{and}\
		M_{P_2} \coloneqq  \sup_{x\in \Omega_{0}}\left\{\|\xi\| : \xi\in \partial
		P_2(x)\right\}, \label{eq-def-MP12}
	\end{align}
	where $\Omega_{0}$ is given in Assumption~\ref{ass-gen} and $\hat{\Omega}_{0}$ is given in \eqref{eq-def-hat Omega}.
	Under Assumptions~\ref{ass-gen} and \ref{ass-sm}, one can see that $M_f$, $M_{P_1}$ and $M_{P_2}$ are all finite.
	In the next theorem, we establish a bound on $\{\lambda_{k+1}\}$ involving these constants.
	\begin{theorem}[{Boundedness of Lagrange multiplier}]
		Consider \eqref{opt-conic-dc} and \eqref{opt-supp-dc}, and suppose that Assumptions~\ref{ass-gen}, \ref{ass-sm}, and \ref{ass-muk} hold.
		Let $\eta$ and $M_{L}$ be given respectively in Lemma~\ref{lm-mfcq-region} and \eqref{eq-bd-MLg},
		$M_{\mathcal{B}}$, $M_{G}$ and $M_f$ be given in \eqref{eq-M-B-def}, \eqref{eq-M-G-def} and \eqref{eq-def-Mf}, respectively,
		and $M_{P_1}$ and $M_{P_2}$ be given in \eqref{eq-def-MP12}.
		Let $\{\lambda_{k+1}\}$ be the
		sequence generated by Algorithm~\ref{alg-SMBA}.
		Then it holds that
		\begin{equation} \label{eq-bd-lambda}
			\lambda_{k+1} \le  \frac{M_f + M_{P_1}+M_{P_2}}{\bar \eta}
			+ \frac{M_{L} M_{G}^2M_{\mathcal{B}}^2
				\mu_{0}}{2\check{L} \bar \eta^2}
			\eqqcolon  M_{\lambda} \quad \forall\, k\in \mathbb{N}_{0},
		\end{equation}
		where
		\begin{equation}\label{eq-def-eta-R}
			\bar \eta =  \sqrt{\mathrm{min}\left\{\eta, 2\alpha_{3}\check{L},
				-2\mu_{0}^{-1}\check{L}g_{\mu_{0}}(x^{0}),
				2c_0\alpha_4\check{L}
				\right\}} > 0
		\end{equation}
        with
        \begin{equation}\label{c0def}
        c_0:=\min\{\mu_{t-1}/\mu_t-1: t = 1,\ldots,\bar t\,\} > 0
        \end{equation}
        and $\bar t$ being the first positive integer such that $\mu_{\bar t} < \eta/(2\alpha_3)$.\footnote{Note that $\mu_t > \mu_{t+1}$ for all $t\in \mathbb{N}_0$ and $\lim_{t\to\infty}\mu_t=0$ as consequences of Assumption~\ref{ass-muk}. These imply $\bar t < \infty$ and $c_0 > 0$.}
	\end{theorem}
	\begin{proof}
		Let $k\in \mathbb{N}_0$.
		For notational simplicity, let $\hat{x}^{k} =  \hat{x}^{k,j_{k}}$ and $R_{k}= R_{k,j_k}$ be defined in
		\eqref{eq-def-center} and \eqref{eq-def-radius}, respectively.
		If $\|x^{k+1} - \hat{x}^{k}\|<R_{k}$, then we have $g_{k,j_{k}}(x^{k+1})<0$ in view of \eqref{eq-sub-const-ball}. Hence,
		we can deduce from \eqref{eq-slack} that
		$\lambda_{k+1} = 0$.
		Consequently, \eqref{eq-bd-lambda} holds.
		
		Now, suppose that $\|x^{k+1} - \hat{x}^{k}\|=R_{k}$.
		Let $M_{R}$ and $\hat \Omega_{0}$ be defined as in Lemma~\ref{lm-g-mu-grad-Lip}.
		Note from \eqref{eq-def-center} that
		$$
		\|\hat{x}^{k} - x^{k}\| = \mu_{k} (L^{k}_{g})^{-1} \| \nabla g_{\mu_{k}}
		(x^{k})\|
		\le \mu_{0} \check{L}^{-1} M_{G} M_{\mathcal{B}} \le M_{R},
		$$
where the first inequality follows from \eqref{eq-bd-MLf}, \eqref{eq-bd-nabla g_mu}, and the fact that $x^k\in \Omega_0$ (see Lemma~\ref{lm-well-defined}(i)). The above display shows that $\hat{x}^{k} \in B(x^{k}, M_R)\subseteq \Omega_{0} + B(0, M_R) = \hat \Omega_{0}$ since $x^{k} \in \Omega_{0}$.
		Next, using the convexity of $P_1$ and the definition of $M_{P_1}$ in \eqref{eq-def-MP12}, we have
		\begin{equation}\label{eq-P1-Lips}
			|P_1(\hat{x}^{k}) - P_1(x^{k+1})| \le M_{P_1} \| x^{k+1} - \hat{x}^{k}\|.
		\end{equation} 	
		Moreover, using \eqref{eq-sub-const-ball} and $\|x^{k+1} - \hat{x}^{k}\| = R_{k}$, we have
		\begin{align}\label{eq-lambda-R-1}
				&\frac{\lambda_{k+1} L_{g}^{k} R_{k}^2 }{\mu_{k}}
                = \frac{\lambda_{k+1} L_{g}^{k} R_{k}^2 }{2\mu_{k}} + \frac{\lambda_{k+1} L_{g}^{k} R_{k}^2 }{2\mu_{k}}\notag\\
				&= - \lambda_{k+1}g_{k,j_{k}}(\hat{x}^{k})
				+ \frac{\lambda_{k+1} L_{g}^{k}}{2\mu_{k}} \|x^{k+1} - \hat{x}^{k}\|^2\notag
				\\& \overset{\mathrm{(a)}}{\le} P_1(\hat{x}^{k}) - P_1(x^{k+1})
				+ \langle \nabla f(x^{k}) - \xi^{k}, \hat x^{k} - x^{k+1}\rangle
				+ \frac{L_{f}^{k}}{2} \| \hat{x}^{k} - x^{k}\|^2\notag
				\\& \overset{\mathrm{(b)}}{\le} (M_{P_1} + \|\nabla f(x^{k})\| + \|\xi^{k}\| )\|x^{k+1}
				- \hat x^{k}\|
				+ \frac{L_{f}^{k}}{2} \| \hat{x}^{k} - x^{k}\|^2\notag
				\\ & \overset{\mathrm{(c)}}{\le}  (M_{f} + M_{P_1} + M_{P_2})R_{k}
				+ \frac{L_f^{k} M_{G}^2M_{\mathcal{B}}^2\mu_{k}^2}{2(L^{k}_{g})^2},
		\end{align}
		where (a) follows from \eqref{eq-lambda-bd-hat x} with $i=i_k$, $j=j_k$ and $x = \hat{x}^{k}$, (b) holds because of \eqref{eq-P1-Lips}, and (c) uses \eqref{eq-def-Mf}, \eqref{eq-def-MP12} and the following relation:
		\begin{equation*}
			\|\hat x^{k} - x^{k}\| = \mu_{k}(L^{k}_{g})^{-1}\|\nabla g_{\mu_{k}}(x^{k})\|
			\le \mu_{k}M_{G}M_{\mathcal{B}}(L^{k}_{g})^{-1};
		\end{equation*}
		here, the equality follows from \eqref{eq-def-center} and the inequality follows
		from \eqref{eq-bd-nabla g_mu}.
		Rearranging the terms in \eqref{eq-lambda-R-1}
		and writing $\Delta_{k} \coloneqq  \mu_{k}^{-1}L_{g}^{k}R_{k}$,
		we get
		\begin{equation}\label{eq-lambda-bd-k}
			\begin{split}
				\lambda_{k+1}
				& \!\le\! \frac{M_f \!+\! M_{P_1}\!+\!M_{P_2}}{\Delta_{k}}
				\!+\! \frac{M_{G}^2M_{\mathcal{B}}^2 L_f^{k}\mu_{k}}{2L^{k}_{g}\Delta_{k}^2}
 \!\le\! \frac{M_f \!+\! M_{P_1}\!+\!M_{P_2}}{\Delta_{k}}
				\!+\! \frac{M_{L} M_{G}^2M_{\mathcal{B}}^2
					\mu_{0}}{2\check{L} \Delta_{k}^2},
			\end{split}
		\end{equation}
		where we use \eqref{eq-bd-MLf} and $\mu_{k}\le\mu_{0}$ to get
		the last inequality.
		
		Finally, we claim that $\Delta_{k}\ge \bar \eta$ for any $k\in \mathbb{N}_{0}$,
		where $\bar \eta$ is given by \eqref{eq-def-eta-R}.
		Granting this, the desired inequality \eqref{eq-bd-lambda} will then follow from
		\eqref{eq-lambda-bd-k}.
		To prove $\Delta_k \ge \bar\eta$, notice from \eqref{eq-def-radius} and the definition of $\Delta_{k}$ that
		\begin{equation}\label{eq-R ge 0}
			\Delta_{k}^2 =  \| \nabla g_{\mu_{k}}(x^{k})\|^2  -
			2\mu_{k}^{-1}L_{g}^{k}g_{\mu_{k}}(x^{k})\quad \forall\, k\in \mathbb{N}_{0}.
		\end{equation}
		Let $k\in\mathbb{N}_{0}$ and $u^{k} =  \nabla h_{\mu_{k}}(G(x^{k}))$.
		Since $h_{\mu_{k}}$ is convex, we have
		\begin{equation}\label{eq-sla-bd-k}
			- \langle u^{k},\, G(x^{k})\rangle
			\le h_{\mu_{k}}(0) - h_{\mu_{k}}(G(x^{k}))
			\le \alpha_{3} \mu_{k} - g_{\mu_{k}}(x^{k}),
		\end{equation}
		where we use $h_{\mu_{k}}(0) \le \sigma_{\mathcal{B}}(0) + \alpha_{3}\mu_{k} =
		\alpha_{3}\mu_{k}$ (see \eqref{eq-sm-appr-bd}) to get the last inequality.

		If $\alpha_{3} \mu_{k} - g_{\mu_{k}}(x^{k}) \le \eta$, then we have $- \langle
		u^{k},\, G(x^{k})\rangle \le \eta$ from \eqref{eq-sla-bd-k}.
		In addition, since $u^{k} \in {\changescolor \partial_{\alpha_{3} \mu_{k}} \sigma_{\mathcal B} (G(x^{k}))} \subseteq \mathcal{B}$ thanks to Proposition~\ref{prop-eps-subdiff} {\changescolor and \eqref{eq-sig-subdiff}}, we conclude in view of Lemma~\ref{lm-mfcq-region} that $\|DG(x^{k})^*u^{k}\|^2 >\eta$.
		Now we can deduce from \eqref{eq-R ge 0} and $g_{\mu_{k}}(x^{k})< 0$
		that
		\begin{equation}\label{eq-R-lb-1}
			\Delta_{k}^2 > \| \nabla g_{\mu_{k}}(x^{k})\|^2 = \|DG(x^{k})^*u^{k}\|^2>\eta.
		\end{equation}
		Next, we consider the case
		\begin{equation}\label{eq-no-cq}
			\alpha_{3} \mu_{k} - g_{\mu_{k}}(x^{k}) > \eta.			
		\end{equation}
		From \eqref{eq-R ge 0} and \eqref{eq-bd-MLf}, we see that
		\begin{equation}\label{eq-mu-R-lb}
			\begin{split}
				\Delta_{k}^2& \ge - 2\mu_{k}^{-1}L_{g}^{k} g_{\mu_{k}}(x^{k})
				\ge - 2\check{L} \mu_{k}^{-1} g_{\mu_{k}}(x^{k}).
			\end{split}
		\end{equation}
		On the one hand, if $\mu_{k} < \frac{\eta}{2\alpha_{3}}$, then the above
		inequality and \eqref{eq-no-cq} give
		\begin{equation}\label{eq-R-lb-2}
			\Delta_{k}^2
			\ge 2\check{L} ( \eta/\mu_{k} - \alpha_{3}) > 2\alpha_{3}\check{L}.
		\end{equation}		
		On the other hand, if $k \ge 1$ and $\mu_{k} \ge \frac{\eta}{2\alpha_{3}}$, then we have $k\le \bar t$ because $\{\mu_k\}$ is decreasing. Hence,
		from \eqref{eq-mu-R-lb}, we have
		\begin{align}\label{eq-R-lb-mu}
				\Delta_{k}^2 & \ge - 2\check{L}\mu_{k}^{-1} g_{\mu_{k}}(x^{k})
				\overset{\mathrm{(a)}}{\ge} - 2\check{L}\mu_{k}^{-1} \left(  g_{\mu_{k-1}}(x^{k}) - \alpha_{4} (\mu_{k-1} - \mu_{k})   \right)
				\notag\\ & \overset{\rm (b)}\ge 2\alpha_{4}\check{L}\mu_{k}^{-1}(\mu_{k-1} - \mu_{k})
				\overset{\mathrm{(c)}}{\ge} 2 c_0\alpha_{4} \check{L},
		\end{align}
		where (a) uses \eqref{eq-gmu-str-des}, (b) uses \eqref{eq-f-des} and (c) uses \eqref{c0def}.
		Finally, letting $k=0$ in \eqref{eq-mu-R-lb}, it holds that
		$$
		\Delta_{0}^2 \ge -2\mu_{0}^{-1}\check{L}g_{\mu_{0}}(x^{0}).
		$$
		Consequently, combining \eqref{eq-R-lb-1}, \eqref{eq-R-lb-2}, \eqref{eq-R-lb-mu}
		and the above display
		with the definition of $\bar \eta$, we conclude $\Delta_{k}\ge \bar \eta$ for
		all $k\in \mathbb{N}_{0}$. 
	\end{proof}
	
	{\changescolor	
		\begin{theorem}[Complexity bound in nonconvex setting]
			\label{th-complexity}
			Consider \eqref{opt-conic-dc} and \eqref{opt-supp-dc}, and suppose that Assumptions~\ref{ass-gen},
			\ref{ass-sm} and \ref{ass-muk} hold. Let $\{x^k\}$ and $\{\lambda_{k+1}\}$
			be generated by Algorithm~\ref{alg-SMBA}.
			Let $M_{\mathcal{B}}$, $M_{L}$ and $M_{\lambda}$ be given in \eqref{eq-M-B-def}, \eqref{eq-bd-MLg} and \eqref{eq-bd-lambda}, respectively.
			Then $G(x^k)\in {\cal K}$ for all $k\in \mathbb{N}_0$ and
			\begin{equation}\label{eq-vk-def}
				v^{k+1} \coloneqq  \lambda_{k+1} \nabla h_{\mu_{k}}(G(x^{k+1})) \in\mathcal{K}^{\circ}\quad \forall\, k\in \mathbb{N}_0.
			\end{equation}
			Moreover, for any $K\in \mathbb{N}_0$, there exists an integer $\hat{k} \in [\lceil K/2\rceil, K]$ such that
			\begin{align}
				& \rho_{{\hat k}+1}
				\le 
				\sqrt{{\bar M}_{1}(\psi(x^{0}) - \psi^*)}S_{K}^{-\frac{1}{2}}, \label{eq-complex-res}
				\\& \|x^{{\hat k}+1} - x^{{\hat k}}\|
				\le  
				\sqrt{2 (\psi(x^{0}) - \psi^*)/\tau_1} \mu_{\lceil K/2\rceil}^{\frac{1}{2}} S_{K}^{-\frac{1}{2}}, \label{eq-complex-suc}
				\\ & - \langle G(x^{{\hat k}+1}), v^{{\hat k}+1}\rangle
				\le \alpha_{3} M_{\lambda} \mu_{\lceil K/2\rceil} 
				+  {\bar M}_2 (\psi(x^{0}) - \psi^*) \mu_{\lceil K/2\rceil} S_{K}^{-1}, \label{eq-complex-slack}
			\end{align}
			where $\psi^*$ is the optimal value of \eqref{opt-conic-dc},
			\begin{align}
				& \rho_{k}\! \coloneqq\! \mathrm{dist} \left(0, \partial P_1(x^{k})\! -\! \partial P_2(x^{k-1})\! +\! \nabla f(x^{k})
				\! + \!
				DG(x^{k})^*v^{k}\right) \ \ \ \forall k\in \mathbb{N},
				\label{eq-rho-def}
				\\ &{\bar M}_1 \coloneqq 16 M_{L}^2 \max\left\{ \mu_{0}/\tau_1,
				M_{\lambda}/\tau_2 \right\} 
				\quad \mbox{and} \quad 
				{\bar M}_2 \coloneqq 
				(M_{L} + \mu_{0}L_{G}M_{\mathcal{B}} ) /\tau_2. \label{eq-def-M-bar}
			\end{align}
		\end{theorem}
        }
		\begin{proof}
			From Proposition~\ref{prop-eps-subdiff} and \eqref{eq-sig-subdiff}, we see that
			$$
			\nabla h_{\mu_{k}}(G(x^{k+1})) {\changescolor\in}\partial_{\alpha_{3} \mu_{k}} \sigma_{\mathcal{B}}(G(x^{k+1})) \subseteq \mathcal{B}\subseteq  \mathcal{K}^{\circ} \quad
			\forall\, k \in \mathbb{N}_0.
			$$
			This proves \eqref{eq-vk-def} as $\lambda_{k+1}\ge 0$. The fact $G(x^k)\in {\cal K}$ was proved in Lemma~\ref{lm-well-defined}(i).

            {\changescolor
			Now, for every $k\in \mathbb{N}_0$, 
			we will bound $\rho_{k+1}$, $\|x^{k+1} - x^{k}\|$
			and $- \langle G(x^{k+1}), v^{k+1}\rangle$ 
			in terms of $\mu_k$ and $\omega_{k}$, where
			\begin{equation}\label{eq-omega-def}
				\omega_k \coloneqq \frac{\tau_1\mu_{k} + \lambda_{k+1} \tau_2}
				{\mu_{k}^2}
				\|x^{k+1} - x^{k}\|^2 \quad \forall k \in \mathbb{N}_{0}.
			\end{equation}
			To this end, by setting $i=i_k$ and $j=j_k$ in \eqref{eq-res}, we have
			\begin{align*}
				& \!-\! \mu_k^{-1}\widetilde L_{k}(x^{k+1}\!-\!x^{k})
				\!+\! \nabla f(x^{k+1}) \!-\! \nabla f(x^{k}) \!+\! \lambda_{k+1} \left(\nabla g_{\mu_{k}}(x^{k+1}) \!-\! \nabla
				g_{\mu_{k}}(x^{k})\right)
				\\ & \in \partial P_1(x^{k+1}) - \xi^{k} + \nabla f(x^{k+1})
				+ \lambda_{k+1} \nabla g_{\mu_{k}}(x^{k+1})
				\\ & \subseteq \partial P_1(x^{k+1}) - \partial P_2(x^{k})
				+ \nabla f(x^{k+1})
				+ \lambda_{k+1} DG(x^{k+1})^* \nabla h_{\mu_{k}}(G(x^{k+1})),
			\end{align*}
			where $\widetilde L_{k} \coloneqq  \mu_{k} L_{f}^{k} + \lambda_{k+1} L_{g}^{k}.$
			Combining this with \eqref{eq-vk-def} and \eqref{eq-rho-def}, we have
			\begin{align*}
				& \rho_{k+1}
			 \le \mu_k^{-1}{\widetilde L_{k}}\|x^{k+1} \!- \! x^{k}\|
			\!	+ \! \|\nabla f(x^{k+1}) \!-\! \nabla f(x^{k})\|\! +\! \lambda_{k+1} \| \nabla g_{\mu_{k}}(x^{k+1}) \!- \! \nabla g_{\mu_{k}}(x^{k})\|
				\\ & \le  \frac{ \mu_{k} (L_f^{k} + L_f) \!+\! \lambda_{k+1} (L_g^{k} + \mu_{k} L_{\mu_{k}})}{\mu_{k}}\|x^{k+1} \!- x^{k}\|
				\le  \frac{2\mu_{k}M_{L} \!+\! 2\lambda_{k+1}M_{L}}{\mu_{k}}\|x^{k+1}\! - x^{k}\|,
			\end{align*}
			where the second inequality follows from the $L_f$-smoothness of $f$ and Lemma~\ref{lm-g-mu-grad-Lip}, and the last inequality holds thanks to \eqref{eq-bd-MLf} (note that $\hat{L} \ge \check{L}$) and the observation that $\mu_k L_{\mu_k}\le \mu_0 L_{\mu_0}$.
			This further yields
			\begin{align}
				\rho_{k+1}^2 
				& \le  \frac{4 M_{L}^2(\mu_{k} +
					\lambda_{k+1})^2}{\mu_{k}^2}
				\|x^{k+1} - x^{k}\|^2 \notag
				\\ &  \overset{\mathrm{(a)}}{=} \frac{ 4 M_{L}^2 (\mu_{k} 
					+ \lambda_{k+1})^2}{\tau_1\mu_{k} + \lambda_{k+1} \tau_2} 
				\omega_k
				\overset{\mathrm{(b)}}{\le}  \frac{8M_{L}^2 (\mu_{k}^2 + \lambda_{k+1}^2)}{\tau_1\mu_{k} +
					\lambda_{k+1} \tau_2} \omega_k 
				\notag\\ & \overset{\mathrm{(c)}}{\le} 8 M_{L}^2 \max\left\{\mu_{k}/\tau_1 ,
				\lambda_{k+1}/\tau_2 \right\} \omega_k
				\overset{\mathrm{(d)}}{\le} 8M_{L}^2 \max\left\{ \mu_{0}/\tau_1,
				M_{\lambda}/\tau_2 \right\}\omega_k,
				\label{eq-res-bd-k}
			\end{align}
			where (a) uses the definition of $\omega_{k}$ in \eqref{eq-omega-def}, (b) follows from the elementary relation $(a+b)^2 \le 2a^2 + 2b^2$ for $a,b\in \R$,
			(c) holds since $t a + (1-t)b \le
			\max\{a, b\}$ for $t\in [0, 1]$ and $a,b\in \mathbb{R}$,
			and (d) is deduced from \eqref{eq-bd-lambda} and the fact that $\mu_k\le \mu_0$.
			
			On the other hand, from \eqref{eq-low-quad-bd}, it follows that
				\begin{align}
				 \lambda_{k+1} g_{\mu_{k}}(x^{k+1}) 
				& \ge  \lambda_{k+1} \left(g_{\mu_{k}}(x^{k}) + \langle \nabla g_{\mu_{k}}(x^{k}),\, x^{k+1} - x^k \rangle
				- \frac{ L_{G}M_{\mathcal{B}}}{2} \|x^{k+1} - x^{k}\|^2 \right) \notag
				 \\ & = \lambda_{k+1} g_{k, j_{k}}(x^{k+1}) 
				 - \frac{ \lambda_{k+1} (L_{g}^{k} + L_{G}M_{\mathcal{B}}\mu_k )}{2 \mu_{k}} \|x^{k+1} - x^{k}\|^2 \notag
				 \\ & = - \frac{\lambda_{k+1}\left(L_{g}^{k} +
				 	\mu_{k}L_{G}M_{\mathcal{B}}\right)}{2(\tau_1\mu_{k} + \lambda_{k+1} \tau_2)} \mu_k\omega_k
				 \ge - \frac{\left(M_{L} +
				 	\mu_{0}L_{G}M_{\mathcal{B}}\right)}{2\tau_2} \mu_k \omega_k, \label{eq-lambda-g-muk}
			\end{align}
			where the first equality holds thanks to the definition of $g_{i,j}$ in \eqref{opt-sub}, 
			the second equality holds because of the complementary slackness condition in \eqref{eq-slack}
			and the definition of $\omega_{k}$ in \eqref{eq-omega-def},
			and the last inequality uses \eqref{eq-bd-MLf}, $\mu_{k}\le\mu_{0}$ and the fact
			$\lambda_{k+1}/(\tau_1\mu_{k} + \tau_2 \lambda_{k+1}) \le \tau_2^{-1}$.
			Now,
			\begin{align}
				& - \langle G(x^{k+1}), v^{k+1}\rangle  = -  \lambda_{k+1}\langle
				G(x^{k+1}), \nabla h_{\mu_{k}}(G(x^{k+1})) \rangle \nonumber
				\\ &\overset{\mathrm{(a)}}{\le} \lambda_{k+1} \left( h_{\mu_k}(0) - h_{\mu_{k}}(G(x^{k+1}) \right)
				\overset{\mathrm{(b)}}{\le} \lambda_{k+1} \left( \alpha_{3}\mu_{k} - g_{\mu_{k}}(x^{k+1})\right) \notag
				\\ & \overset{\mathrm{(c)}}{\le}
				\alpha_{3}M_{\lambda}\mu_{k}
				+ \frac{\left(M_{L} +	\mu_{0}L_{G}M_{\mathcal{B}}\right)}
				{2\tau_2} \mu_k \omega_k,\label{eq-sla-comple-bd}
			\end{align}
			where (a) uses the convexity of $h_{\mu_{k}}$, (b) uses \eqref{eq-sm-appr-bd}, (c) uses \eqref{eq-lambda-g-muk} and \eqref{eq-bd-lambda}.
			
			Furthermore, from the definition of $\omega_k$ in \eqref{eq-omega-def}, it holds that
			\begin{equation}\label{eq-x-bd}
				\|x^{k+1} - x^{k}\|^2 = \frac{\mu_{k}^2\omega_k}{\tau_1\mu_{k} 
					+ \lambda_{k+1} \tau_2}
				\le \frac{\mu_{k}\omega_k}{\tau_1}.
			\end{equation}
			
			Finally, from \eqref{eq-f-des} and the definition of $\omega_{k}$ in \eqref{eq-omega-def}, it follows that
			\begin{equation*}
				\mu_{k} \omega_{k} \le 2 (\psi(x^{k}) - \psi(x^{k+1})) \quad \forall k\in \mathbb{N}_{0}.
			\end{equation*}
			Summing the last inequality from $\lceil K/2\rceil$ to $K$
			and rearranging terms, we have for some $\hat{k} \in [\lceil K/2\rceil, K]$ that
			\begin{equation*}
				\omega_{\hat{k}} = \min_{\lceil K/2\rceil\le k \le K} \omega_{k} 
				\le \frac{2(\psi(x^{\lceil K/2\rceil}) - \psi(x^{K+1}))}{\sum_{k=\lceil K/2\rceil}^{K} \mu_{k}}
				\le 2(\psi(x^{0}) - \psi^*) S_{K}^{-1},
			\end{equation*}
			where the last inequality holds since $\{\psi(x^{k})\}$ is nonincreasing (see \eqref{eq-f-des}).
			Combining this with $\mu_{\hat{k}} \le \mu_{\lceil K/2\rceil}$, \eqref{eq-res-bd-k}, \eqref{eq-sla-comple-bd} and \eqref{eq-x-bd}, we conclude the theorem.
            }
		\end{proof}
	
	{\changescolor	
	\begin{corollary}\label{corol-complex-1}
		Consider \eqref{opt-conic-dc} and \eqref{opt-supp-dc}, and suppose that Assumptions~\ref{ass-gen},
		\ref{ass-sm} and \ref{ass-muk} hold. Let $\{x^k\}$
		be generated by Algorithm~\ref{alg-SMBA}.
		Let $\epsilon_1>0$ and $\epsilon_2>0$. Suppose that $K\in \mathbb{N}_0$ satisfies
		\begin{equation}\label{eq-S-mu-K-cond}
			S_{K} \ge \bar{M}_1 (\psi(x^{0}) - \psi^*) \epsilon_1^{-2}
			\ \  \mbox{and} \ \ 
			\mu_{\lceil K/2\rceil} \le 
			\min \left\{  \frac{{\bar M}_1}{\alpha_{3} {\bar M}_1 M_{\lambda}  
			\! + \!  {\bar M}_2  \epsilon_1^2}, \frac{ \tau_1 {\bar M}_1} { 2 }  \right\} \epsilon_2,
		\end{equation}
		where $\psi^*$ is the optimal value of \eqref{opt-conic-dc}, $M_{\lambda}$ is given by \eqref{eq-bd-lambda},
		and $\bar{M}_1,\bar{M}_2$ are given by \eqref{eq-def-M-bar}.
		Then there exists an integer $\hat{k} \in [\lceil K/2 \rceil, K]$ such that
		$x^{\hat{k}+1}$ is an $(\epsilon_1, \epsilon_2, \epsilon_1 \sqrt{\epsilon_2})$-KKT point of \eqref{opt-conic-dc}.
	\end{corollary}
	\begin{proof}
		By Theorem~\ref{th-complexity}, 
		there exists an integer $\hat{k} \in [\lceil K/2 \rceil, K]$ such that
		\eqref{eq-complex-res}, \eqref{eq-complex-suc} and \eqref{eq-complex-slack} hold.
		Using this fact and \eqref{eq-S-mu-K-cond},
		we have
		\begin{align*}
			& \rho_{{\hat k}+1}
			\le 
			\sqrt{{\bar M}_{1} (\psi(x^{0}) - \psi^*)}S_{K}^{-\frac{1}{2}} \le \epsilon_{1},
			\\ & \|x^{{\hat k}+1} - x^{{\hat k}}\|
			\le \sqrt{2/(\tau_1 {\bar M}_{1}) } \mu_{\lceil K/2\rceil}^{\frac{1}{2}} \epsilon_1 \le \epsilon_1 \sqrt{\epsilon_2},			
			\\ & - \langle G(x^{{\hat k}+1}), v^{{\hat k}+1}\rangle
			\le  (\alpha_{3} M_{\lambda}  
			+  {\bar M}_2 {\bar M}_1^{-1} \epsilon_1^2   ) \mu_{\lceil K/2\rceil} 
			\le \epsilon_2,
		\end{align*}
		where $\{v^{k}\}$ and $\{\rho_{k}\}$ are defined by \eqref{eq-vk-def} and \eqref{eq-rho-def}, respectively.
		This completes the proof.
	\end{proof}
	\begin{corollary}\label{cor-complex-ncov}
		Consider \eqref{opt-conic-dc} and \eqref{opt-supp-dc}, and suppose that Assumptions~\ref{ass-gen} and
		\ref{ass-sm} hold. 
		Let $r\in (0,1)$ and $\mu_0 > 0$, and let
		$
		\mu_{k} = \mu_{0} (k+1)^{-r}$ for all $k \in \mathbb{N}$.
		Let $\{x^k\}$
		be generated by Algorithm~\ref{alg-SMBA}.
		Let $\epsilon_1>0$ and $\epsilon_2>0$. Suppose that $K\in \mathbb{N}_0$ satisfies
		\begin{align*}
			 K \ge \max & \left\{   
			\left(\frac{\bar{M}_1 (\psi(x^{0}) - \psi^*)2^{2r+1}}{\mu_0 }\right)^{\frac{1}{1-r}} \epsilon_1^{-\frac{2}{1-r}}, \right.
			\\ &  
			\left. 2 \mu_0^{\frac{1}{r}} \min  \left\{  \frac{{\bar M}_1}{\alpha_{3} {\bar M}_1 M_{\lambda}  
				+  {\bar M}_2  \epsilon_1^2}, \frac{ \tau_1 {\bar M}_1} {2}  \right\}^{-\frac{1}{r}} \epsilon_2^{-\frac{1}{r}}
				\right\},
		\end{align*}
where $\psi^*$ is the optimal value of \eqref{opt-conic-dc}, $M_{\lambda}$ is given by \eqref{eq-bd-lambda},
and $\bar{M}_1,\bar{M}_2$ are given by \eqref{eq-def-M-bar}.
		Then there exists an integer $\hat{k} \in [\lceil K/2 \rceil, K]$ such that
		$x^{\hat{k}+1}$ is an $(\epsilon_1, \epsilon_2, \epsilon_1 \sqrt{\epsilon_2})$-KKT point of \eqref{opt-conic-dc}.
	\end{corollary}
	\begin{proof}
	The conclusion follows directly from Corollary~\ref{corol-complex-1} and Proposition~\ref{prop-muk}.
	\end{proof}
	\begin{remark}
		\begin{enumerate}[{\rm (i)}]
			\item By choosing $r=\frac{1}{3}$ in Corollary~\ref{cor-complex-ncov}, 
			Algorithm~\ref{alg-SMBA} outputs 
			an $(\epsilon_1, \epsilon_1, \epsilon_1^{3/2})$-KKT point of \eqref{opt-conic-dc}
			within $O(\epsilon_1^{-3})$ outer iterations.
			Similarly, setting $r=\frac{1}{2}$ in Corollary~\ref{cor-complex-ncov} yields
			an $(\epsilon_1, \epsilon^2_1, \epsilon^2_1)$-KKT point of \eqref{opt-conic-dc}
			within $O(\epsilon_{1}^{-4})$ outer iterations.
			
			\item 
			When $\mathbb{X} = \mathbb{R}^{n}$, $\mathbb{Y}=\mathbb{R}^{m}$ and $\mathcal{K}=\mathbb{R}^{m}_{-}$ in \eqref{opt-conic-dc}, 
			the inexact MBA in \cite{25LPB} attains an $(\epsilon_1, \epsilon^2_1, \epsilon_1)$-KKT point within 
			$O\left(\varepsilon^{-2}_{1}\right)$ outer iterations. The total number of inner iterations for solving the subproblems, however, was not taken into account explicitly in their complexity analysis.
            \item When $\mathbb{X} = \mathbb{R}^{n}$, $\mathbb{Y}=\mathbb{R}^{m}$, $\mathcal{K}=\mathbb{R}^{m}_{-}$ and $P_2 = 0$ in \eqref{opt-conic-dc}, the inexact LCPG method in \cite{25BDL} with the constraint extrapolation method \cite{23BDL} as subproblem solver obtains 
            an $(\epsilon_1^2, \epsilon_1^2)$ type-II KKT point in the sense of \cite[Definition~4]{25BDL}\footnote{\changescolor
            An $(\epsilon_1^2, \epsilon_1^2)$ type-II KKT point $x$ must satisfy $\|x-w\| \le \epsilon_1$ for some $(\epsilon_1, \epsilon_1^2, 0)$-KKT point $w$. 
            } within $O(\epsilon_1^{-2})$ outer iterations (see \cite[Theorem~8, Corollary~5]{25BDL}), 
            and $O(\epsilon_1^{-4})$ (resp., $O(\epsilon_1^{-8})$) inner iterations with (resp., without) knowing a bound on the norm of a Lagrange multiplier (see \cite[Corollaries 4 and 5]{25BDL}).
            
			\item In contrast, although the outer iteration complexity of our proposed {\em s}MBA method for computing an $(\epsilon_1, \epsilon, \epsilon_1^{3/2})$-KKT point of \eqref{opt-conic-dc} is $O(\epsilon_1^{-3})$, each subproblem in {\em s}MBA involves only one single inequality constraint and can potentially be solved at a lower computational cost. We will further comment on this aspect in section~\ref{sec6}. However, we also need to point out that both our approach and the approaches in \cite{25BDL,25LPB} require one evaluation of $DG(x^k)$ per outer iteration, meaning that our approach potentially requires more number of access to $DG(\cdot)$, and may become less efficient if the cost of computing $DG(\cdot)$ is high. 
			\end{enumerate}		
	\end{remark}	
}

	\section{{\em s}MBA for convex optimization}\label{sec-conv}
In this section, we study the convergence properties of Algorithm~\ref{alg-SMBA} for convex instances of \eqref{opt-conic-dc}.
	\begin{assumption}\label{ass-conv}
		In \eqref{opt-conic-dc}, $P_2 = 0$, $f$ is convex, and the mapping $G$ is $(-\mathcal{K})$-convex.			
	\end{assumption}

	\begin{lemma}
		Consider \eqref{opt-conic-dc} and \eqref{opt-supp-dc}, and suppose that Assumptions~\ref{ass-gen},
		\ref{ass-sm}, \ref{ass-muk} and \ref{ass-conv} hold. Let $\{x^k\}$ be generated by Algorithm~\ref{alg-SMBA}.
		Let $M_{L}$ and $M_{\lambda}$ be given in \eqref{eq-bd-MLg} and \eqref{eq-bd-lambda}, respectively.
		Then, we have, for every $k \in \mathbb{N}_0$ and
		any $\tilde{x}\in \Omega^*$,
		\begin{equation}\label{eq-ball-prox-conv}
			\begin{split}
				& \frac{\mu_{k}}{M_1} \left( \psi(x^{k+1}) \!-\! \psi^*\right) \le \|x^{k} \!-\! \tilde x\|^2 \!-\! \|x^{k+1} \!-\! \tilde x\|^2 \!+\! M_2 \mu_{k}^2 \!+\! M_3 \left(\psi(x^{k}) \!-\! \psi(x^{k+1})\right),	
			\end{split}
		\end{equation}
		where  $\Omega^*$ and $\psi^*$ are the solution set and the optimal value of \eqref{opt-conic-dc}, respectively, and
		\begin{equation}\label{eq-M1-2-3}
 M_1 {\changescolor\, \coloneqq} M_{L}(\mu_{0} + M_{\lambda})/2,\ M_2 {\changescolor\, \coloneqq} 2 \alpha_{3} / \check{L},\
 M_3 {\changescolor\, \coloneqq} 2 \max\{0, L_f -\check{L}\}/(\tau_1 \check{L}).
		\end{equation}
	\end{lemma}
	\begin{proof}
		Let $\{\lambda_{k+1}\}$, $\{L^{k}_{f}\}$ and $\{L^{k}_{g}\}$ be generated by Algorithm~\ref{alg-SMBA}.
		By setting $i=i_k$ and $j=j_k$ in \eqref{eq-ball-prox} with $P_2 = 0$, we have,
		for any $x\in{\mathbb X}$,
		\begin{align}\label{eq-des-conv-1}
				\psi(x^{k+1}) \!-\! \psi(x)
				&\le f(x^{k}) \!-\! f(x) + \langle \nabla f(x^{k}), x \!-\! x^{k}\rangle +
				\lambda_{k+1} g_{k,j_k}(x)
				\notag\\ & \  \quad \!+\! \frac{L^{k}_{f}}{2}\|x\!-\!x^{k}\|^2
				\!+\! \frac{L_{f} \!-\! L^{k}_{f}}{2}\|x^{k+1} \!-\! x^{k}\|^2
				\!-\! \frac{{\changescolor \widetilde{L}}_{k}}{2\mu_{k}}  \|x^{k+1} \!-\!
				x\|^2
				\notag\\ &  \le \lambda_{k+1} g_{k,j_k}(x)
				\!+\! \frac{L^{k}_{f}}{2}\|x\!-\!x^{k}\|^2
				\notag\\& \quad \ \!+\! \frac{\max\{0, L_f -\check{L}\}}{2}\|x^{k+1} \!-\! x^{k}\|^2
				\!-\! \frac{{\changescolor \widetilde{L}}_{k}}{2\mu_{k}}  \|x^{k+1} \!-\!
				x\|^2,
		\end{align}
		where ${\changescolor \widetilde{L}_{k} \coloneqq} \mu_{k} L^{k}_{f} + \lambda_{k+1} L^{k}_{g}$,
		and the last inequality holds because of the convexity of $f$ and the fact
		$
		L_f - L_{f}^{k} \le L_f - \check{L} \le \max\{0, L_f -\check{L}\}.
		$
		Next, by Assumption~\ref{ass-conv} and Proposition~\ref{prop-g-mu}(iv), $g_{\mu_{k}}$ is convex.
		Hence, for any $x\in G^{-1}(\mathcal{K})$,
		\begin{align*}
				g_{k,j_k}(x)  & = g_{\mu_{k}}(x^{k}) + \langle \nabla g_{\mu_{k}}(x^{k}),\, x -
				x^{k} \rangle  + \frac{L^{k}_{g}}{2\mu_{k}}\|x-x^{k}\|^2
				\\  & \le g_{\mu_{k}}(x) + \frac{L^{k}_{g}}{2\mu_{k}}\|x-x^{k}\|^2
				\le \alpha_{3} \mu_{k}  + \frac{L^{k}_{g}}{2\mu_{k}}\|x-x^{k}\|^2,
		\end{align*}
		where the equality follows from the definition of $g_{k,j_k}$ in
		\eqref{opt-sub} and the last inequality
		follows from \eqref{eq-gmu-appr} and the fact that $\sigma_{\mathcal{B}}(G(x)) \le 0$.		
		Combining this with \eqref{eq-des-conv-1}, we have, for all $x\in G^{-1}(\mathcal{K})$,
		\begin{align}\label{eq-des-M_+}
				\psi(x^{k+1}) - \psi(x)
				& \le \alpha_{3} \lambda_{k+1} \mu_{k} + \frac{\max\{0, L_f -\check{L}\}}{2}\|x^{k+1} - x^{k}\|^2
				\notag\\ &\quad + \frac{{\changescolor \widetilde{L}}_{k}}{2\mu_{k}} \left(\|x^{k} -
				x\|^2 - \|x^{k+1} - x\|^2\right).
		\end{align}
		
Now, from \eqref{eq-bd-MLf}, \eqref{eq-bd-lambda} and the facts $\mu_{k}\le \mu_{0}$ and $\lambda_{k+1}\ge 0$, we have
		\begin{align}
			& \mu_{k}\check{L} \le {\changescolor \widetilde{L}}_{k} = \mu_{k} L^{k}_{f} + \lambda_{k+1} L^{k}_{g} \le \mu_{0} M_{L} + M_{\lambda} M_{L} = 2 M_1, \label{eq-tilde L-bd}\\
			& \lambda_{k+1}/{\changescolor \widetilde{L}}_{k} \le (L^{k}_{g})^{-1} \le \check{L}^{-1}. \label{eq-lambda-tilde L}
		\end{align}
		From the second inequality in \eqref{eq-tilde L-bd}
		and the fact $\psi(x^{k+1}) - \psi^* \ge 0$, it follows that
		\begin{align*}
				& \frac{\mu_{k}}{M_1}
				\left( \psi(x^{k+1}) - \psi^*\right)
				\le \frac{2\mu_{k}}{{\changescolor \widetilde{L}}_{k}}
				\left( \psi(x^{k+1}) - \psi^*\right)
				\\ & \overset{\mathrm{(a)}}{\le} \frac{2\alpha_{3} \lambda_{k+1}}{{\changescolor \widetilde{L}}_{k}}\mu_{k}^2
				+ \frac{\mu_{k}\max\{0, L_f -\check{L}\}}{{\changescolor \widetilde{L}}_{k}}\|x^{k+1} - x^{k}\|^2
				+ \|x^{k} - \tilde x\|^2 - \|x^{k+1} - \tilde x\|^2		
				\\ & \overset{\mathrm{(b)}}{\le} \frac{2\alpha_{3}}{\check{L}} \mu_{k}^2
				+ \frac{\max\{0, L_f -\check{L}\}}{\check{L}}\|x^{k+1} - x^{k}\|^2
				+ \|x^{k} - \tilde x\|^2 - \|x^{k+1} - \tilde x\|^2
				\\&\overset{\mathrm{(c)}}{\le} \frac{2\alpha_{3}}{\check{L}} \mu_{k}^2
				+ \frac{2\max\{0, L_f -\check{L}\}}{\tau_1 \check{L}} \left(\psi(x^{k}) - \psi(x^{k+1})\right)
				+ \|x^{k} - \tilde x\|^2 - \|x^{k+1} - \tilde x\|^2,	
		\end{align*}
		where (a) holds by substituting $x=\tilde{x} \in \Omega^*$ in \eqref{eq-des-M_+}, (b) uses \eqref{eq-lambda-tilde L} and the first inequality in \eqref{eq-tilde L-bd}, and (c) follows from \eqref{eq-f-des}.
	\end{proof}
	
We now establish the convergence of the $\{x^k\}$ generated by {\em s}MBA and its iteration complexity for function values; recall from Lemma~\ref{lm-well-defined}(i) that $G(x^k)\in {\cal K}$ for all $k$.
	\begin{theorem}[Complexity bound in convex setting]\label{th-complexity-conv}
		Consider \eqref{opt-conic-dc} and \eqref{opt-supp-dc}, and suppose that Assumptions~\ref{ass-gen},
		\ref{ass-sm}, \ref{ass-muk} and \ref{ass-conv} hold. Let $\{x^k\}$ be generated by Algorithm~\ref{alg-SMBA}.
		Define $M_1,\, M_2$ and $M_3$ by \eqref{eq-M1-2-3}.
		Then, we have, for every $K \in \mathbb{N}_0$,
		\begin{equation}\label{eq-complexity-conv}
				\psi(x^{K+1}) \!-\! \psi^* \!\le\! \frac{M_1}{\sum_{k={\changescolor \lceil K/2\rceil}}^{K}\mu_{k}} \left(M_2\!\!\!\sum_{k={\changescolor \lceil K/2\rceil}}^{K}\!\!\!\!\mu_{k}^2
				+ M_3\left(\psi(x^{0}) - \psi^*\right) + {\changescolor{\rm diam}(\Omega_0)}^2\right)\!\!,
		\end{equation}
		where $\Omega^*$ and $\psi^*$ are the solution set and the optimal value of \eqref{opt-conic-dc}, respectively{\changescolor, and ${\changescolor{\rm diam}(\Omega_0)}:= \sup_{x,y\in \Omega_0}\|x - y\|<\infty$}.\footnote{\changescolor Notice that ${\changescolor{\rm diam}(\Omega_0)}$ is finite because $\Omega_0$ is bounded. Also, note that the right hand side of \eqref{eq-complexity-conv} converges to zero as $K\to\infty$ thanks to Assumption~\ref{ass-muk}, since this assumption implies $\sum_{k=\lceil K/2\rceil}^{K}\mu_{k}^2 \le \mu_{\lceil K/2\rceil}\sum_{k=\lceil K/2\rceil}^{K}\mu_{k}$, $ \mu_{\lceil K/2\rceil}\to 0$ and $\sum_{k=\lceil K/2\rceil}^{K}\mu_{k}\to \infty$.}
		If we assume further that $\sum_{k=0}^{\infty}\mu_{k}^2 < \infty$, then $\{x^{k}\}$ converges to some $x^*\in \Omega^*$.
	\end{theorem}
	\begin{proof}
		Let $K\in \mathbb{N}_0$.
		For any $\tilde{x}\in \Omega^*$, summing \eqref{eq-ball-prox-conv} from {\changescolor$k=\lceil K/2\rceil$} to $K$, we get
		\begin{equation*}
			\begin{split}
				& \textstyle \left(\psi(x^{K+1}) - \psi^*\right) \sum_{k={\changescolor \lceil K/2\rceil}}^{K}\mu_{k}
				\le \sum_{k={\changescolor \lceil K/2\rceil}}^{K}\mu_{k} \left( \psi(x^{k+1}) - \psi^*\right)
				\\ & \textstyle\le M_1 \left(M_2 \sum_{k={\changescolor \lceil K/2\rceil}}^{K}\mu_{k}^2
				\!+\! M_3 \left(\psi({\changescolor x^{\lceil K/2\rceil}}) - \psi(x^{K+1})\right)
				\!+\! \|{\changescolor x^{\lceil K/2\rceil}} \!-\! \tilde x\|^2 \!-\! \|x^{K+1} \!-\! \tilde x\|^2\right)
				\\ & \textstyle\le M_1 \left(M_2 \sum_{k={\changescolor \lceil K/2\rceil}}^{K}\mu_{k}^2
				+ M_3 \left(\psi(x^{0}) - \psi^*\right)
				+ {\changescolor{\rm diam}(\Omega_0)}^2\right),	
			\end{split}
		\end{equation*}
where {\changescolor the first inequality holds} as $\psi(x^k)\ge \psi(x^{k+1})$ (see \eqref{eq-f-des}){\changescolor, and we used $\psi(x^k)\ge \psi(x^{k+1})$ and $x^k\in \Omega_0$ (see Lemma~\ref{lm-well-defined}(i)) in the last inequality}. This proves \eqref{eq-complexity-conv}.
		
		Now, suppose further that $\sum_{k=0}^{\infty}\mu_{k}^2 < \infty$. First, we {\changescolor know} from {\changescolor Theorem~\ref{th-complexity} that some cluster point of $\{x^k\}$ belongs to $\Omega^*$. In addition, from} \eqref{eq-ball-prox-conv}, it holds that
		\begin{equation*}
			\|x^{k+1} - \tilde x\|^2 \le \|x^{k} - \tilde x\|^2
			+ M_2 \mu_{k}^2
			+ M_3 (\psi(x^{k}) - \psi(x^{k+1})) \quad \forall\, k\in \mathbb{N}_0,\, \tilde{x} \in \Omega^*.
		\end{equation*}
		Note that $\psi(x^{k}) - \psi(x^{k+1})$ is non-negative (see \eqref{eq-f-des}) and summable (since $\psi^*>-\infty$ thanks to Assumption~\ref{ass-gen}).
		Then, according to \cite[Proposition 1]{03Iusem}, we conclude that $\{x^{k}\}$ converges to some $x^*\in \Omega^*$.
	\end{proof}

We next discuss local convergence rate. In the literature, local convergence rate of first-order methods is usually studied by assuming the so-called Kurdyka-{\L}ojasiewicz (KL) property, and the explicit convergence rate is closely related to a quantity known as the KL exponent; see, e.g., \cite{10ABRS,13ABS,18LP}. Here, we consider the growth condition in \eqref{eq-KL-theta} below, which is closely related to the KL property with exponent $\theta\in (0,1)$; see \cite[Theorem~5]{17BNPS} and \cite[Lemma~3.10]{21YPL}. In particular, condition \eqref{eq-KL-theta} with $\theta=\frac12$ can be satisfied by many application models; see, e.g., \cite[Section~4.1]{21YPL}. In addition to \eqref{eq-KL-theta}, we also need an additional control on $\{\mu_k\}$, as stated in the next assumption.
	\begin{assumption}\label{ass-mu-2}
		There exist ${\bar r}\in (\frac{1}{2}, 1)$, $c_1>0$, $c_2>0$ and $\bar k_0\in \mathbb{N}_{0}$
		such that
		\begin{equation}\label{eq-mu-delta}
			c_1 (k+1)^{-{\bar r}} \le \mu_{k} \le c_2 (k+1)^{-{\bar r}} \quad \forall\, k\ge \bar k_0.
		\end{equation}
	\end{assumption}
One can check that if we define $\{\mu_k\}$ as in \eqref{eq-def-muk} and impose in addition that $r_k \equiv \bar r$ for some $\bar r \in (\frac12,1)$, then $\{\mu_k\}$ satisfies both Assumptions~\ref{ass-muk} and \ref{ass-mu-2}.

We are now ready to present our local convergence rate result. The induction argument in our proof is motivated by the proof of {\changescolor\cite[Lemma~4.2]{24LMX}}.
	
	\begin{theorem}[Local convergent rate under a growth condition]\label{th-rate-KL}
Consider \eqref{opt-conic-dc} and \eqref{opt-supp-dc}, and suppose that Assumptions~\ref{ass-gen}, \ref{ass-sm}, \ref{ass-muk}, \ref{ass-conv} and \ref{ass-mu-2} hold.
		Let $\Omega^*$ and $\psi^*$ be the solution set and the optimal value of \eqref{opt-conic-dc}, respectively.
		Assume there exist $\kappa>0$, $\theta\in(0, 1)$, $\epsilon_0>0$ and $\epsilon_1\in (0,1)$
		such that
		\begin{equation}\label{eq-KL-theta}
			\mathrm{dist}(x, \Omega^*) \le \kappa (\psi(x) - \psi^*)^{1-\theta}
		\end{equation}
		for all $x\in \Omega_0$ with $\mathrm{dist}(x, \Omega^*)\le
		\epsilon_0$
		and $\psi(x) \le \psi^* + \epsilon_1$.
		Let $\{x^k\}$ be generated by Algorithm~\ref{alg-SMBA}.
		Then $x^* := \lim_{k\to \infty}x^k$ exists, and there exist $k_1\in \mathbb{N}_0$ and $\kappa_1 >0$ such that
		\begin{equation}\label{eq-rate-conv}
			\|x^{k} - x^*\| \le \kappa_1\cdot (k+1)^{-s} \quad \forall\, k\ge k_1,
		\end{equation}
		where
		\begin{equation}\label{eq-rate-s}
			s \coloneqq s(\bar r) \coloneqq \begin{cases}
				{\bar r} - \frac{1}{2}
				& \mbox{if } \theta \in (0, \frac{1}{2}],\\
				\min\left\{{\bar r} - \frac{1}{2}, \frac{(1-{\bar r})(1-\theta)}{2\theta-1}\right\}
				& \mbox{if } \theta\in (\frac{1}{2}, 1).
			\end{cases}
		\end{equation}
	\end{theorem}
	\begin{remark}[Concerning the bound \eqref{eq-rate-conv}]
		Suppose that Assumptions~\ref{ass-gen}, \ref{ass-sm}, \ref{ass-conv} and the condition \eqref{eq-KL-theta} hold.
		Let
\begin{equation}\label{choicemu}
r \in (0.5,1),\ \mu_0 > 0\ {\rm and}\ \mu_k = \mu_0 (k+1)^{-r}\ \ \forall k\in \mathbb{N}.
\end{equation}
		Then {\changescolor$\max\{1,\sum_{k=\lceil K/2\rceil}^{K}\mu_{k}^2\}/\sum_{k=\lceil K/2\rceil}^{K}\mu_{k} = O(K^{-(1-r)})$}, and $\{\mu_k\}$ takes the form of \eqref{eq-def-muk} with $r_k \equiv r$, $n_0 = 0$ and $\nu_0 = 1$; in particular, Assumption~\ref{ass-muk} is satisfied with this choice of $\{\mu_k\}$ thanks to Proposition~\ref{prop-muk}.
		Using these together with Theorem~\ref{th-complexity-conv} and \eqref{eq-KL-theta}, one can see that
		$
		\mathrm{dist}(x^{K}, \Omega^*) = O\left(K^{-\bar s}\right)$,
		where $\bar s\coloneqq\bar s(r) \coloneqq (1 - \theta)(1 - r)$.
		Clearly, we have $\sup_{r\in(\frac{1}{2}, 1)} \bar s(r) = \frac{1-\theta}{2}$.
		
		On the other hand, for the choice of $\{\mu_k\}$ in \eqref{choicemu}, one can also apply Theorem~\ref{th-rate-KL} to deduce that
		$\mathrm{dist}(x^{K}, \Omega^*) \le \|x^{K}-x^*\|= O(K^{-s})$,
		where $s \coloneqq s(r)$ is defined in \eqref{eq-rate-s}.
		Note that
		\begin{equation}\label{eq-sup-s}
			\sup_{r\in(\frac{1}{2}, 1)} s(r) = \begin{cases}
				\frac{1}{2} & \mbox{if }\theta \in (0, \frac{1}{2}],
				\\ \frac{1 - \theta}{2\theta} & \mbox{if } \theta\in (\frac{1}{2}, 1),
			\end{cases}
		\end{equation}
		where $\sup_{r\in(\frac{1}{2}, 1)} s(r) = s(\frac{1}{2\theta})$ when $\theta\in (\frac{1}{2}, 1)$.
		From \eqref{eq-sup-s}, one can see that
		$
		\textstyle\sup_{r\in(\frac{1}{2}, 1)} \bar s(r) < \sup_{r\in(\frac{1}{2}, 1)} s(r).
		$
		In other words, for the choice of $\{\mu_k\}$ in \eqref{choicemu}, by leveraging the explicit knowledge of $\theta$ in choosing $r$, the local convergence rate guaranteed by Theorem~\ref{th-rate-KL} is better than that obtained by directly applying Theorem~\ref{th-complexity-conv} and \eqref{eq-KL-theta}.

        {\changescolor Finally, notice that the local convergence rate guaranteed by Theorem~\ref{th-rate-KL} appears to be worse than that derived in \cite{21YPL} for the case when $\mathbb{X} = \mathbb{R}^{n}$, $\mathbb{Y}=\mathbb{R}^{m}$ and $\mathcal{K}=\mathbb{R}^{m}_{-}$. Here, we need to point out that the complexity per (outer) iteration in \cite{21YPL} is different from that of our algorithm. Indeed, \cite{21YPL} assumed that the subproblems of their algorithm are solved exactly, while in general (say, when $m > 1$) these problems can only be solved inexactly and require an iterative solver; see \cite[Corollaries 4 and 5]{25BDL} for explicit complexity analysis involving inexactly solved subproblems. 
In contrast, our {\em s}MBA subproblems involve only one single inequality constraint, 
and can be solved efficiently via some simple one-dimensional root finding procedures when the proximal mapping of $\gamma P_1$ can be computed efficiently for all $\gamma > 0$.}
	\end{remark}
	
	\begin{proof}[{Proof of Theorem~\ref{th-rate-KL}}]
		By Assumption~\ref{ass-mu-2}, one can see that $\sum_{k=0}^{\infty}\mu_{k}^2 < \infty$. Hence, $\lim_{k\to \infty}x^{k} = x^* \in \Omega^*$ by Theorem~\ref{th-complexity-conv}.
		For simplicity, we write, for all $ k\in \mathbb{N}_{0}$,
		\begin{align}
\textstyle
			\delta_{k} \coloneqq \psi(x^{k}) - \psi^*,\quad \gamma_k \coloneqq  \sum_{i=k}^{\infty} \mu_{i}^2,
\quad \beta_{k}\coloneqq \|x^{k} - x^*\|^2
			+  M_2\gamma_{k} +  M_3 \delta_{k},\label{eq-delta-gamma-def}
		\end{align}
where $M_2$ and $M_3$ are defined in \eqref{eq-M1-2-3},
and we note the following simple observation for $\gamma_k$: since ${\bar r} > 1/2$, from \eqref{eq-delta-gamma-def} and \eqref{eq-mu-delta}, we have, for every integer $k \ge \bar k_0$,
		\begin{equation}\label{eq-bar gamma-bd}
\textstyle				{\gamma}_{k+1}  \le c_2^2 \sum_{i=k+1}^{\infty} (i+1)^{-2{\bar r}}
				\le c_2^2 \int_{k}^{\infty} (t+1)^{-2{\bar r}} dt
= c_2^2 (2{\bar r}-1)^{-1}(k+1)^{-(2{\bar r}-1)}.
		\end{equation}

Now, substituting $\tilde{x} = x^*$ in \eqref{eq-ball-prox-conv}, we get, for every $k\in \mathbb{N}_0$,
		\begin{equation}\label{eq-recur}
			\begin{split}
				M_1^{-1} \mu_{k}\delta_{k+1}
				&
				\le  M_2 \mu_{k}^2
				+  M_3 \left(\psi(x^{k}) - \psi(x^{k+1})\right) + \|x^{k} - x^*\|^2 - \|x^{k+1} - x^*\|^2 \\
&= \beta_{k} - \beta_{k+1}.
			\end{split}
		\end{equation}
		On the other hand, since $x^k\to x^*\in \Omega^*$, there exists an integer $k_{0}\ge 0$ such that
		\begin{equation}\label{eq-KL-region}
			\mathrm{dist}(x^{k}, \Omega^*) \le \epsilon_0
			\quad \mbox{and}\quad \delta_{k} = \psi(x^{k}) - \psi^* \le \epsilon_1 < 1 \quad \forall\, k \ge k_0.
		\end{equation}	
		This together with \eqref{eq-KL-theta} and the fact $x^k\in \Omega_0$ (see Lemma~\ref{lm-well-defined}(i)) for all $k\in \mathbb{N}_0$ gives
		\begin{equation}\label{eq-KL-k0}
			\mathrm{dist}(x^{k}, \Omega^*) \le \kappa (\psi(x^{k}) - \psi^*)^{1-\theta} = \kappa \delta_{k}^{1-\theta}\quad \forall\, k \ge k_0.
		\end{equation}
		

Next, we note from \eqref{eq-ball-prox-conv} that
		\begin{equation*}
			\|x^{k+1} - \tilde x\|^2 \le \|x^{k} - \tilde x\|^2
			+ M_2 \mu_{k}^2
			+ M_3 (\psi(x^{k}) - \psi(x^{k+1})) \quad \forall\, k\in \mathbb{N}_0,\, \tilde{x} \in \Omega^*.
		\end{equation*}
Consequently, for any $k \in \mathbb{N}_0$, $k^{\prime}\in \mathbb{N}$ and all $\tilde{x}\in \Omega^*$, we have
		\begin{equation}\label{eq-x-k-k'-tild x}
			\begin{split}
				\| x^{k+k^{\prime}} - \tilde{x} \|^2
				& \le \|x^{k+k^\prime-1} - \tilde{x} \|^2
				+ M_2 \mu_{k+k^\prime-1}^2
				+ M_3 (\psi(x^{k+k^\prime-1}) - \psi(x^{k+k^\prime}))
				\\ & \textstyle \le \| x^{k} - \tilde{x} \|^2
				+ M_2 \sum_{i=k}^{k+k^{\prime}-1}\mu_{i}^2
				+ M_3 (\psi(x^{k}) - \psi(x^{k+k^{\prime}}))
				\\ & \le \| x^{k} - \tilde{x} \|^2
				+ M_2 \gamma_{k}
				+ M_3 \delta_{k}.
			\end{split}
		\end{equation}
		Setting $\tilde{x} = \bar x^{k} \in \Omega^*$ with $\|x^{k} - \bar x^{k}\| = \mathrm{dist}(x^{k}, \Omega^*)$ in the last inequality, we obtain
		\begin{equation}\label{eq-xk-xk-prime}
			\begin{split}
				\|x^{k} - x^{k+k^{\prime}} \|^2
				&\le 2 \left(\|x^{k} - \bar x^{k} \|^2
				+ \| x^{k+k^{\prime}} - \bar x^{k} \|^2\right)
				\\& \le 4 \left(\mathrm{dist}(x^{k}, \Omega^*)^2
				+M_2 \gamma_{k}
				+ M_3 \delta_{k}
				\right) \quad\ \quad \forall\, k \in \mathbb{N}_0,\ k^{\prime}\in \mathbb{N}.
			\end{split}
		\end{equation}
		Thus, fixing $k\in \mathbb{N}_{0}$ and passing to the limit as $k^{\prime} \to \infty$ in \eqref{eq-xk-xk-prime}, we obtain
		\begin{equation}\label{eq-x-x*-dist-bd}
			\begin{split}
				&\|x^{k} - x^{*} \|^2  \le 4 \left(\mathrm{dist}(x^{k}, \Omega^*)^2
				+M_2 \gamma_{k}
				+ M_3 \delta_{k}
				\right) \quad \forall k\in \mathbb{N}_{0}.
			\end{split}
		\end{equation}
		Now, it follows that, for every $k\in \mathbb{N}_0$,
		\begin{equation}\label{eq-beta-bd}
			\begin{split}
				\beta_{k+1} & = \|x^{k+1} - x^*\|^2
				+ M_3 \delta_{k+1}
				+ M_2\gamma_{k+1}
 \le \|x^{k+1} - x^*\|^2
				+ M_3 \delta_{k}
				+ M_2\gamma_{k}
				\\ & \le \|x^{k} - x^{*}\|^2 + 2 M_3 \delta_{k}
				+ 2 M_2\gamma_{k} \le 2 \beta_{k},
			\end{split}
		\end{equation}
		where the first inequality follows from \eqref{eq-f-des}
		and \eqref{eq-delta-gamma-def},
		and the second inequality follows from \eqref{eq-x-k-k'-tild x} with $k^\prime = 1$ and $\tilde{x} = x^*\in \Omega^*$.
		
		From the definition of $\beta_{k}$ in \eqref{eq-delta-gamma-def} and \eqref{eq-x-x*-dist-bd}, it follows that, for every $k\ge k_0$,
		\begin{align}\label{eq-beta-bd-psi}
				 \beta_{k+1} & \!\le\! 4 \mathrm{dist}(x^{k+1}, \Omega^*)^2  \!+\! 5 M_3 \delta_{k+1}
				\!+\! 5 M_2 \gamma_{k+1}
\!\overset{\mathrm{(a)}}{\le}\! 4\kappa^2 \delta_{k+1}^{2(1-\theta)} \!+\! 5 M_3 \delta_{k+1} \!+\! 5 M_2 \gamma_{k+1}
				\notag\\& = 5 M_2 \gamma_{k+1} \!+\!
				\begin{cases}
					\left(4\kappa^2 \delta_{k+1}^{1-2\theta} + 5 M_3\right) \delta_{k+1} & \mbox{if } \theta\in (0,\frac{1}{2}],
					\\ \left(4\kappa^2  + 5 M_3 \delta_{k+1}^{2\theta-1}\right)\delta_{k+1}^{2(1-\theta)} & \mbox{if }\theta\in (\frac{1}{2}, 1),
				\end{cases}
				\notag\\ & \overset{\mathrm{(b)}}{\le} c_4 \delta_{k+1}^{1/p} + 5 M_2 \gamma_{k+1},
		\end{align}
		where (a) uses \eqref{eq-KL-k0}, and (b) holds {\changescolor thanks to the fact $\delta_{k+1} \le \epsilon_1$ (see \eqref{eq-KL-region}) upon defining} 
		\begin{equation}\label{def-p-c4}
p{\changescolor\,\coloneq} \max\left\{1,\frac1{2(1-\theta)}\right\}\ {\rm and}\
			c_4 {\changescolor\,\coloneq} \begin{cases}
				4\kappa^2 \epsilon_1^{1-2\theta} + 5M_3 & \mbox{if}\ \theta \in (0, \frac{1}{2}],
				\\
				4\kappa^2 + 5M_3 \epsilon_1^{2\theta - 1} & \mbox{if}\ \theta \in (\frac{1}{2}, 1).
			\end{cases}
		\end{equation}		
		
Now, define $b_1 {\changescolor\coloneq} 1/(c_4^{p}M_1)$, $b_2 {\changescolor\coloneq} 5  p M_2 2^{p-1}/(c_4^{p}M_1)$, and $k_1$, $C_1$ and $C_0$ by
		\begin{align}
			& k_1 {\changescolor\,\coloneq} \begin{cases}
				\max\left\{k_0, \bar k_0, \left\lceil \left( \frac{(2{\bar r} -1)4^{{\bar r}}}{b_1 c_1}\right)^{\frac{1}{1-{\bar r}}}\right\rceil - 1 \right\} & \mbox{if } \theta\in(0, \frac{1}{2}],\\
				\max\{k_0, \bar k_0\} & \mbox{if } \theta \in (\frac{1}{2}, 1),
			\end{cases}
			\label{eq-k_1-def}\\
			& C_1 {\changescolor\,\coloneq} \begin{cases}
				\max\left\{C_0, \frac{b_2c_2^3 2^{2{\bar r}}}{(2{\bar r} - 1)b_1c_1} \right\}
				& \mbox{if } \theta\in(0, \frac{1}{2}],\\
				\max\left\{C_0, \frac{4^{sp+1}b_2 c_2^3}{b_1 c_1 (2{\bar r} - 1)}, \left(\frac{s4^{sp+1}}{b_1c_1}\right)^{\frac{1}{p-1}} \right\}
				& \mbox{if } \theta \in (\frac{1}{2}, 1),
			\end{cases} \label{eq-C_1-def-ii}\\
			& C_0 {\changescolor\,\coloneq} \max\{\bar \beta (k_1 + 1)^{2s}, 5 c_2^2 (2{\bar r}-1)^{-1} M_2 2^{2{\bar r}-1}\}, \notag
		\end{align}
		where $s$ is defined in \eqref{eq-rate-s}, $c_1$, $c_2$, ${\bar r}$ and $\bar k_0$ are given in Assumption~\ref{ass-mu-2}, and $\bar\beta := \sup_{k\in\mathbb{N}_0} \beta_{k}$, which is finite because $\lim_{k\to\infty}\beta_{k} = 0$.\footnote{From the definition of $\beta_k$ in \eqref{eq-delta-gamma-def}, Theorem~\ref{th-complexity-conv} and
		the fact $\sum_{k=0}^{\infty}\mu_{k}^2<\infty$ (see Assumption~\ref{ass-mu-2}), one can see that $\lim_{k\to\infty}\beta_{k} = 0$.}
		
We claim that for all $k\ge k_1$,
		\begin{equation}\label{eq-comp-beta-i}
			\beta_{k} \le C_1 (k+1)^{-2s}.
		\end{equation}
		Granting this, we can see from \eqref{eq-delta-gamma-def} that \eqref{eq-rate-conv} holds with $\kappa_1 = \sqrt{C_1}$.
		
We prove \eqref{eq-comp-beta-i} holds for all $k\ge k_1$ by induction.		
First, from \eqref{eq-C_1-def-ii}, we have
		$\beta_{k_1} \le \bar \beta \le C_0 (k_1 + 1)^{-2s} \le C_1 (k_1 + 1)^{-2s}$,
		which implies that \eqref{eq-comp-beta-i} holds for $k=k_1$.
		Next, suppose that \eqref{eq-comp-beta-i} holds for some $k\ge k_1$.
		If $\beta_{k+1} \le 5M_2 \gamma_{k+1}$, then from \eqref{eq-bar gamma-bd}, we have
		\begin{align*}
				& \beta_{k+1}  \le 5c_2^2 (2{\bar r}-1)^{-1} M_2(k+1)^{-(2{\bar r}-1)} \\
				& \le  5c_2^2 (2{\bar r}-1)^{-1} 2^{2{\bar r}-1} M_2(k+2)^{-(2{\bar r}-1)}
				{\changescolor \le C_0 (k + 2)^{-(2\bar r - 1)}}\le C_1 (k + 2)^{-2s},
		\end{align*}
		where the second inequality holds because $2{\bar r}-1>0$ and $k+1\ge \frac{1}{2}(k+2)$, and the last inequality follows from \eqref{eq-rate-s} and \eqref{eq-C_1-def-ii}.
		
		Now, suppose that $\beta_{k+1} \ge 5M_2 \gamma_{k+1}$. Note that the function $t \mapsto t^{p}$ is convex for $t\in \mathbb{R}_{+}$ since $p\ge 1$ (see \eqref{def-p-c4}).
		From this, \eqref{eq-delta-gamma-def} and \eqref{eq-beta-bd-psi}, we have, as $k\ge k_1\ge k_0$,
		\begin{align*}
				& \psi(x^{k+1}) - \psi^* = \delta_{k+1}
				\ge (c_4^{-1} \beta_{k+1}  - 5 c_4^{-1} M_2 \gamma_{k+1})^{p}
				 \\& \ge  (c_4^{-1} \beta_{k+1})^{p}
				+ p (c_4^{-1} \beta_{k+1})^{p - 1}
				( - 5 c_4^{-1} M_2 \gamma_{k+1})
				 = c_4^{-p} \beta_{k+1}^{p}
				- 5 p c_4^{-p}M_2 \gamma_{k+1} \beta_{k+1}^{p-1}.
		\end{align*}
		Combining this with \eqref{eq-recur}, we get,
		\begin{align}\label{eq-betak}
				\beta_{k} - \beta_{k+1} &\textstyle\ge \frac{\mu_{k}}{M_1}\delta_{k+1}
				\ge \frac{1}{c_4^{p}M_1} \mu_{k}\beta_{k+1}^{p} - \frac{5 p M_2}{c_4^{p}M_1} \mu_{k}\gamma_{k+1} \beta_{k+1}^{p-1}\nonumber
				\\ & \ge b_1 \mu_{k}\beta_{k+1}^{p} - b_2 \mu_{k}\gamma_{k+1} \beta_{k}^{p-1},
		\end{align}
		where the last inequality follows from \eqref{eq-beta-bd}.
		
		Define
		\begin{equation}\label{eq-def-phi-beta}
			\phi_{k}(t) \coloneqq t + b_1 {\mu}_{k} t^{p}\quad \forall\, t\ge 0,
		\end{equation}
		In what follows, we will show that
		\begin{equation}\label{eq-psi-beta}
			\phi_{k}(\beta_{k+1}) \le \phi_k( C_1 (k+2)^{-2s}),
		\end{equation}
		which will imply that \eqref{eq-comp-beta-i} holds for $k+1$ since $\phi_{k}$ is strictly increasing on
		$\mathbb{R}_{+}$.
		This will complete the induction and the conclusion of this theorem follows.
		
		We now prove \eqref{eq-psi-beta}. Since $k \ge k_1 \ge k_0$, from \eqref{eq-def-phi-beta} and \eqref{eq-betak}, we have
		\begin{equation*}
			\phi_{k}(\beta_{k+1}) = \beta_{k+1} + b_1 {\mu}_{k} \beta_{k+1}^{p}\le \beta_{k} + b_2{\mu}_{k} \gamma_{k+1} \beta_{k}^{p-1},
		\end{equation*}
		from which we deduce that
		\begin{align}\label{eq-phi-s-r_1}
				 &\phi_k(C_1  (k+2)^{-2s}) - \phi_{k}(\beta_{k+1})
				\!\ge\! \frac{C_1}{(k+2)^{2s}} + \frac{b_1 C_1^{p}\mu_{k}}{(k+2)^{2sp}}   - \beta_{k} - b_2{\mu}_{k} \gamma_{k+1} \beta_{k}^{p-1}
				\notag\\ & \overset{\mathrm{(a)}}{\ge} \frac{C_1}{(k+2)^{2s} }
				+ \frac{b_1 \mu_{k} C_1^{p}}{ (k+2)^{2sp}}   -  \frac{C_1}{(k+1)^{2s}}
				- \frac{b_2{\mu}_{k} \gamma_{k+1} C_1^{p-1}}{(k+1)^{2s(p-1)}}
				\notag\\ & \overset{\mathrm{(b)}}{\ge} \frac{b_1C_1^p\mu_k}{4^{s p}(k+1)^{2sp}}
				-
				\frac{2sC_1}{(k+1)^{2s+1}}
				-
				\frac{b_2\mu_k\gamma_{k+1}C_1^{p-1}}{(k+1)^{2s(p-1)}}
				\notag\\ & \overset{\mathrm{(c)}}{\ge} \frac{b_1c_1C_1^p}{4^{s p}(k+1)^{2s p + {\bar r}}}
				-
				\frac{2sC_1}{(k+1)^{2s+1}}
				-
				\frac{b_2c_2^3C_1^{p-1}}{(2{\bar r}-1)(k+1)^{2s(p-1)+3{\bar r}-1}}
				\notag\\ & = \frac{1}{(k+1)^{\widetilde r_1}}\Bigl(
				\frac{b_1c_1C_1^p}{4^{sp}}
				-\frac{2sC_1}{(k+1)^{\widetilde r_2}}
				-\frac{b_2c_2^3C_1^{p-1}}{(2{\bar r} - 1)(k+1)^{\widetilde r_3}}
				\Bigr),
		\end{align}
		where $\widetilde r_1 = 2s p + {\bar r}$, $\widetilde r_2 = 1 - {\bar r} - 2s(p-1)$, $\widetilde r_3 = 2{\bar r} - 2s - 1$, (a) uses the induction hypothesis,
		(b) uses $k+2\le 2(k+1)$ and $(k+2)^{-2s} - (k+1)^{-2s} \ge - 2s(k+1)^{-(2s+1)}$ (thanks to the convexity of $t\mapsto t^{-2s}$ for $t> 0$),
		and (c) follows from \eqref{eq-mu-delta} and \eqref{eq-bar gamma-bd}.
		
		We consider two cases: $\theta \in (0, \frac{1}{2}]$ and $\theta \in (\frac{1}{2}, 1)$.
		
		\noindent {\bf Case (i)}. $\theta\in(0,\frac{1}{2}]$.
		In this case, $p=\max\{1, \frac{1}{2(1-\theta)}\} = 1$ and $s={\bar r} - \frac{1}{2}$ thanks to
		\eqref{eq-rate-s}.
		Then $\widetilde r_2 = 1- {\bar r}>0$ and $\widetilde r_3 = 0$, which together with \eqref{eq-phi-s-r_1} gives
		\begin{equation*}
			\begin{split}
				 \phi_k\left(C_1  (k+2)^{-2s}\right) - \phi_{k}(\beta_{k+1})
				& \ge \frac{1}{(k+1)^{\widetilde r_1}}\Bigl(
				\frac{b_1c_1C_1}{4^{{\bar r}-\frac{1}{2}}}
				-\frac{(2{\bar r} -1)C_1}{(k+1)^{1-{\bar r}}}
				-\frac{b_2c_2^3}{2{\bar r} - 1}
				\Bigr)
				\\ & \ge \frac{1}{(k+1)^{\widetilde r_1}}\Bigl(
				\frac{b_1c_1C_1}{4^{{\bar r}}}
				-\frac{(2{\bar r} -1)C_1}{(k+1)^{1-{\bar r}}}
				\Bigr)\ge 0,
			\end{split}
		\end{equation*}
		where we use \eqref{eq-C_1-def-ii}
		and \eqref{eq-k_1-def} to get the last two inequalities, respectively.
		
		\noindent {\bf Case (ii)}. $\theta \in (\frac{1}{2}, 1)$.
		In this case, $p \!=\! \frac{1}{2(1-\theta)}\!>\! 1$ and
		$
		s \!= \!\min\left\{{\bar r} \! - \! \frac{1}{2}, \frac{1-{\bar r}}{2(p-1)}\right\}
		$.
		This shows $\widetilde r_2 \!=\! 1 \!-\! {\bar r} \!-\! 2s(p \!-\! 1) \!\ge\! 0$ and $\widetilde r_3 \!=\! 2{\bar r} \!-\! 2s \!-\! 1 \!\ge\! 0$, which, together with \eqref{eq-phi-s-r_1}, gives
		\begin{align*}
				\phi_k\left(C_1  (k+2)^{-2s}\right) - \phi_{k}(\beta_{k+1})
			 &\ge \frac{1}{(k+1)^{\widetilde r_1}}\Bigl(
				\frac{b_1c_1C_1^p}{4^{sp}}
				-2sC_1
				-\frac{b_2c_2^3C_1^{p-1}}{2{\bar r} - 1} \Bigr)
				\\  & \ge \frac{1}{(k+1)^{\widetilde r_1}}\Bigl(
				\frac{b_1c_1C_1^p}{2^{2sp+1}}
				-2sC_1
				\Bigr)\ge 0,
		\end{align*}
		where the second inequality holds because $C_1 \ge \frac{4^{sp+1}b_2 c_2^3}{b_1 c_1 (2{\bar r} - 1)}$ and the last inequality holds because $C_1 \ge [(s4^{sp+1})/(b_1c_1)]^{1/(p-1)}$, in view of the definition of $C_1$ in \eqref{eq-C_1-def-ii}.
	\end{proof}
	{\changescolor
	\section{Numerical experiments}\label{sec5}
	In this section, we present a simulation study of Algorithm~\ref{alg-SMBA} and compare its performance against CVX (version 2.2) with SDPT3 (version 4.0) as the solver.\footnote{All codes are written and executed in Matlab {\changescolor R2025b} on a machine equipped with a 12th Gen Intel(R) Core(TM) i7-12700 CPU at 2.10 GHz, 32 GB of RAM, and running Windows 11 Enterprise. {\changescolor The codes are available at \href{https://github.com/XuJiefeng-CN/Smoothing-MBA-NSDP}{https://github.com/XuJiefeng-CN/Smoothing-MBA-NSDP}.}}
We consider the following convex $\ell_1$-regularized NSDP:
    \begin{align}\label{opt-nsdp}
    \begin{array}{rl}
        \min\limits_{x\in \R^n} & \psi(x) \coloneqq 
\sum_{i=1}^{n}\left( \frac{1}{4}d_i x_i^4 + \frac{1}{3}c_i |x_i|^3\right) + 
\frac{1}{2}x^\top Q x + b^\top x + \sum_{i=1}^n|x_i|
\\
\text{s.t.} &  G(x) \coloneqq -A_0 - \sum_{i=1}^n x_i A_i \in \mathcal{S}_{-}^{m},
\end{array}
    \end{align}
where $b\in \mathbb{R}^{n}$, $c\in \mathbb{R}^{n}_{+}$, $d\in \mathbb{R}^{n}_{+}$,
$Q\in \mathcal{S}^{n}_{+}$, and $A_i\in \mathcal{S}^{m}_{+}$ (with $m\ge 2$), $i=0,1, \ldots, n$. 

We generate random instances of \eqref{opt-nsdp} as follows. First, we set $Q = U \mathrm{Diag}(a) U^{\top}$ and $A_i = U_i \mathrm{Diag}(a^{i}) U_i^{\top}$ for $i=0,1,\ldots,n$,
where $U\in \mathbb{R}^{n\times n}$ and $U_i\in \mathbb{R}^{m\times m}$ are orthogonal matrices generated respectively by the Matlab commands \verb|qr(randn(n))| and \verb|qr(randn(m))|, and $a \in \mathbb{R}^{n}_{+}$ and $a^i \in \mathbb{R}^{m}_{+}$ are random vectors.\footnote{\changescolor Here, for a vector $y$, we let ${\rm Diag}(y)$ denote the diagonal matrix with the $i$th diagonal entry being $y_i$ for all $i$.}
The vectors $c,d, a, a^i$, $i=1,\ldots, n$, are generated independently using the Matlab command 
\verb|100*sprand(n,1,0.2)|, while the entries of $a^0$ are generated independently from the uniform distribution on $[10, 100]$.
Finally, define $\bar a \in \mathbb{R}^{n}$ by setting
$\bar a_i = 1$ if $a_i \neq 0$ and $\bar a_i = 0$ otherwise for each $i$.
We then set $b = U \mathrm{Diag}(\bar a) \bar{b}$,
where $\bar{b}$ has i.i.d. Gaussian entries with mean $10$ and variance $1$.

For these random instances, note that the origin is strictly feasible for \eqref{opt-nsdp} since $A_0$ is positive definite.
Moreover, we have by construction that $x\mapsto \frac{1}{2}x^\top Q x + b^\top x$ is bounded below on $\mathbb{R}^{n}$,
which implies that $\psi$ in \eqref{opt-nsdp} is level bounded. 
Consequently, if we set $x^0 = 0$,
\begin{equation}\label{Arguhaha}
P_2(x) \coloneqq 0,\ P_1(x) \coloneqq \textstyle \sum_{i=1}^n|x_i|,\ f(x) \coloneqq \psi(x) - P_1(x)\ \ \forall x \in \mathbb{R}^{n},
\end{equation}
then Assumption~\ref{ass-gen}(i), (iii), (iv), and (v) are satisfied, while Assumption~\ref{ass-gen}(ii) is violated as $\nabla f$ is not globally Lipschitz. Despite this, we argue that a direct application of Algorithm~\ref{alg-SMBA} is still well-defined as long as the MSA satisfies Assumption~\ref{ass-sm}. Indeed, as long as $x^k\in \Omega_0$, proceeding as in the display before  \eqref{eq-g-ls}, we see that the ``trial points" $x^{k,i,j} \in \Omega_{0}+B(0,M_R)$ for all nonnegative integers $i, j$; in particular, if we define $\Xi_0:= \{x:\; G(x)\in {\cal S}^m_-,\ \frac{1}{2}x^\top Q x + b^\top x + \sum_{i=1}^n|x_i|\le \psi(x^0)\}$, then $\Xi_0$ is compact and $x^{k,i,j} \in \Xi_{0}+B(0,\widetilde M_R)$ for all nonnegative integers $i, j$, where $\widetilde M_R$ is defined in a way similar to $M_R$ in \eqref{eq-def-MR} but with $g^*_{\cal B}$ replaced by $\inf_{x\in \Xi_0}g_{\cal B}(x)$ and $M_G$ replaced by $\sup\{\|DG(x)\|:\; x\in \Xi_0\}$. Consequently, using an induction argument, one can see that applying {\em s}MBA directly to \eqref{opt-nsdp} (with $f$, $P_1$ and $P_2$ as in \eqref{Arguhaha}) is the same as applying {\em s}MBA to minimizing $\tilde \psi \coloneqq \tilde f + P_1-P_2$ over the same constraint set of \eqref{opt-nsdp}, where $\tilde f$ is a convex globally Lipschitz differentiable function that agrees with $f$ in \eqref{Arguhaha} on a {\em bound open box (centered at the origin) containing $\Xi_{0}+B(0,\widetilde M_R)$} and has $x_i\mapsto x_i^4$ and $x_i\mapsto |x_i|^3$ extended Lipschitz differentiably (and monotonically) outside the box; furthermore, one can check that Assumption~\ref{ass-gen} holds for the latter minimization problem. The well-definedness now follows upon applying Lemma~\ref{lm-well-defined} to this latter optimization problem; moreover, our convergence results are applicable if we choose $\{\mu_k\}$ accordingly.
   
{\bf Setting of parameters of Algorithm~\ref{alg-SMBA}.}
	Let $\mathcal{B} = \{u\in \mathcal{S}^{m}_{+} :  {\rm tr}(u) = 1\}$.
	Then $\sigma_{\cal B}(y) = \lambda_{\rm max}(y) $ for all $y\in \mathcal{S}^{m}$.
	From Remark~\ref{rem25}, an MSA of $\sigma_{\cal B}$ can be chosen as $\{h_{\mu}\}$ with $h_{\mu}(y):= \mu \log \left(\sum_{i=1}^{m}e^{\lambda_i(y)/\mu}\right) + 10^{-5} \mu$ for $y\in \mathcal{S}^{m}$ and $\mu>0$.

    Let $n_0 = 300$, $\nu_0=1/(10 n_0+1)$ and $K=5000$.
    Let $\bar r \in (0.01, 1)$ and $\bar s\ge 0$.
    For each $k\in \mathbb{N}_{0}$, let $k_1$, $k_2\in \mathbb{N}_{0}$ satisfy $k = k_2 (n_0+1) + k_1$ and $k_1\le n_0$, 
    and let $\bar k \coloneqq k_2 (n_0 +1) + \nu_{0} k_1$.
    We consider $\{\mu_{k}\}$
    as follows: set $\mu_0 >0$ and define
		\begin{equation}\label{eq-def-muk-2}
			\mu_{k} = \mu_{0}\left( \bar k +
				1\right)^{-r_{\bar k}} (\log( \bar k +
				3 ))^{-s_{\bar k}} \quad \forall\, k\in \mathbb{N},
		\end{equation}
where $r_{j} \coloneqq 0.01 + \min\{1, \frac{j}{K}\} (\bar r - 0.01)$ 
and $s_{j} \coloneqq \min\{1, \frac{j}{K}\} \bar s$
for all $j\in \mathbb{N}_0$.\footnote{\changescolor Note that when $\bar s = 0$, $\{\mu_k\}$ defined in \eqref{eq-def-muk-2}
    reduces to the one defined in \eqref{eq-def-muk}.
Moreover, since $\{s_k\}$ in \eqref{eq-def-muk-2} is nondecreasing and $\log(\cdot+3)$ is an increasing positive-valued function on $\mathbb{N}_0$, we can see from Proposition~\ref{prop-muk} that the $\{\mu_k\}$ defined in \eqref{eq-def-muk-2} is positive, diminishing, and decreasing. Moreover, we also have that for any 
$\nu > 0$, $(\log(\bar k + 3))^{-\bar s} = \Omega(k^{-\nu})$, which together with Proposition~\ref{prop-muk} allows us to show that for any 
$ \hat r\in (\bar r,1)$, $\sum_{k=\lceil K/2\rceil}^K \mu_k = \Omega(K^{-(1-\hat r)})$. Consequently, the
     $\{\mu_k\}$ defined in \eqref{eq-def-muk-2} also satisfies Assumption~\ref{ass-muk}.
    }

We solve subproblem \eqref{opt-sub} through the root-finding
	scheme in \cite[Appendix~A]{21YPL}.
	We let $\tau_1 = 0.01$, {\changescolor $\tau_2 = 0.01$}, $\check{L} = 10^{-8}$ and $\hat{L} = 10^{8}$.
	We let $x^0 = 0$ and $\mu_0 = 0.9(2^{-l_{0}}) $ with $l_{0}$ being the first nonnegative integer such that $g_{\mu_{0}}(x^{0}) \le 0.1 g_{\mathcal{B}}(x^{0})$. For \( k = 0 \), we choose \( L^{k,0}_f = L^{k,0}_g = 1 \).
To choose $L^{k,0}_f$ and $L^{k,0}_g$ for $k\ge 1$, we consider \( \Delta x \coloneqq x^{k} - x^{k-1} \), \( \Delta f \coloneqq \nabla f(x^{k}) - \nabla f(x^{k-1}) \) and \( \Delta g \coloneqq \mu_k(\nabla g_{\mu_{k}}(x^k) - \nabla g_{\mu_{k-1}}(x^{k-1})) \).
	If $\|\Delta x\| > 10^{-12}$ and $L_{f}^{\mbox{\footnotesize BB}} \coloneqq  |\langle \Delta x, \Delta f \rangle| / \|\Delta x\|^2 \in [\check{L}, \hat L]$,
	then let $L^{k,0}_f = L_{f}^{\mbox{\footnotesize BB}}$; otherwise let $L^{k,0}_f = \min\{\hat L,\max\{\check{L}, \frac{1}{2} L^{k-1,0}_f\}\}$.
	Similarly, if $\sqrt{|\langle \Delta x, \Delta g \rangle|}> 10^{-12}$ and $L_{g}^{\mbox{\footnotesize BB}} \coloneqq \|\Delta g\|^2/|\langle \Delta x, \Delta g \rangle| \in [\check{L}, \hat L]$,
	then let $L^{k,0}_g = L_{g}^{\mbox{\footnotesize BB}}$; otherwise let $L^{k,0}_g = \min\{\hat L,\max\{\check{L}, \frac{1}{2} L^{k-1,0}_g\}\}$.
	We terminate when
    \begin{equation}\label{termination_criterion}
    \frac{\sqrt{\tau_1\mu_{k} + \lambda_{k+1} \tau_2}}
				{\mu_{k}}
				\frac{\|x^{k+1} - x^{k}\|}{\max\{1,\|x^{k+1}\|\}} \le \epsilon\quad \mbox{and} \quad 
                - \frac{\langle G(x^{k+1}), v^{k+1} \rangle}{\max\{1,\|x^{k+1}\|\}} \le \epsilon,
    \end{equation}
    where $\epsilon>0$ and $v^{k+1}$ is defined by \eqref{eq-vk-def}; the first criterion above is motivated by the definition of $\omega_k$ in \eqref{eq-omega-def}.
    The algorithm is also terminated if the number of iterations exceeds $5000$
    or the $j$ in {\bf Step 3} exceeds $40$.

First, we study the effect of different choices of $\bar r$ and $\bar s$, which govern the rate of decrease of $\{\mu_{k}\}$.
We set $\epsilon = 10^{-7}$ in \eqref{termination_criterion}.
Figure~\ref{fig-1} illustrates the performance of {\em s}MBA with different rates of decrease ($\bar r\in \{0.33, 0.6, 0.9\}$ and $\bar s \in 
\{0, 3\}$) on two problems of different dimensions, namely $(n,m)= (500,500)$ and $(1000, 500)$.
Among the {\em s}MBA variants, we see that larger $\bar r$ and $\bar s$ lead to faster rates of decrease of the objective value.

	\begin{figure}[h]
		\label{fig-1}
		\caption{\changescolor
        Testing {\em s}MBA on different sets of $(\bar r,\bar s)$. Here, $\omega_k:= (\psi(x^k) - \psi^{\rm cvx})/\max\{1,|\psi^{\rm cvx}|\}$, where $\psi^{\rm cvx}$ is the objective value output by CVX. The time for CVX includes the time for problem automatic reformulation.
        }
		\centering
		\subfigure{\changescolor
			\begin{overpic}[width=0.46\textwidth]{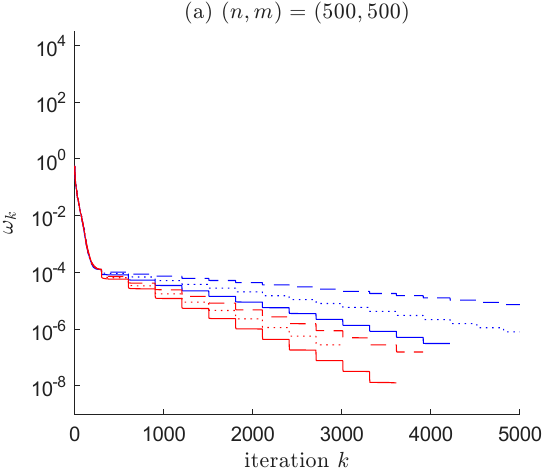}
            \put(25, 64){
					\scalebox{.55}{
						\begin{tabular}{cccc}
							\hline
							line style  &  $(\bar r, \bar s)$ & objective & \text{time (s)} \\
                            \hline
   {\color{blue}\bf - - - } & $(0.33, 0)$ & $ -120.107594$ & $  443.1$\\
{\color{blue}\bf $\cdots$ } & $( 0.6, 0)$ & $ -120.108368$ & $  431.3$\\
     {\color{blue}\bf --- } & $( 0.9, 0)$ & $ -120.108427$ & $  352.9$\\
    {\color{red}\bf - - - } & $(0.33, 3)$ & $ -120.108446$ & $  329.6$\\
 {\color{red}\bf $\cdots$ } & $( 0.6, 3)$ & $ -120.108431$ & $  253.7$\\
      {\color{red}\bf --- } & $( 0.9, 3)$ & $ -120.108462$ & $  296.5$\\
                             &   cvx  & $ -120.108464$ & $ 2040.3$  \\
							\hline
						\end{tabular}		
					}
				}
			\end{overpic}
		}
		\subfigure{\changescolor
			\begin{overpic}[width=0.46\textwidth]{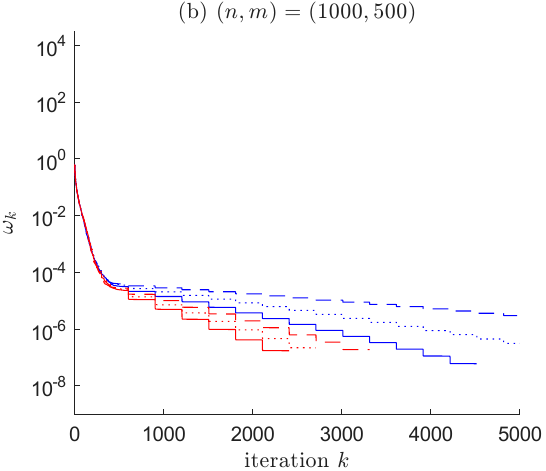}
				\put(25, 64){
					\scalebox{.55}{
						\begin{tabular}{cccc}
							\hline
							line style  &  $(\bar r, \bar s)$ & {\changescolor objective} & \text{time (s)} \\
                            \hline
 {\color{blue}\bf - - - } & $(0.33, 0)$ & $  -316.11054$ & $  753.8$\\
{\color{blue}\bf $\cdots$ } & $( 0.6, 0)$ & $  -316.11139$ & $  750.8$\\
     {\color{blue}\bf --- } & $( 0.9, 0)$ & $  -316.11147$ & $  676.1$\\
    {\color{red}\bf - - - } & $(0.33, 3)$ & $  -316.11143$ & $  500.3$\\
 {\color{red}\bf $\cdots$ } & $( 0.6, 3)$ & $  -316.11142$ & $  411.2$\\
      {\color{red}\bf --- } & $( 0.9, 3)$ & $  -316.11144$ & $  369.3$\\
                             &   cvx  & $  -316.11149$ & $ 6186.0$ \\
                            \hline
						\end{tabular}		
					}
				}
			\end{overpic}
		}
	\end{figure}
    
Next, we fix $\bar r = .9$ and $\bar s = 3$, and test Algorithm~\ref{alg-SMBA} on a range of dimensions $m\in \{100, 500\}$ and $n\in \{100, 500, 1000\}$ with $\epsilon \in \{10^{-5}, 10^{-7}\}$ in \eqref{termination_criterion}.
The average numerical results for each $(m,n)$ over $30$ random instances are summarized in Table~\ref{table-nm}, 
where `iter.' denotes the number of iterations required by {\em s}MBA.\footnote{\changescolor
We observe that the outputs of all {\em s}MBA variants and CVX are feasible for all the test instances; therefore, feasibility violations are not reported in the table.}
One can notice that {\em s}MBA significantly outperforms CVX in terms of CPU time (in seconds), and does not compromise too much in terms of objective values.


       	\begin{table}[h]\label{table-nm}
		\caption{\changescolor Comparisons between {\em s}MBA and CVX, averaged over $30$ random instances for each $(m,n)$. The time for CVX includes the time for problem automatic reformulation.}
		\centering
        {\footnotesize\changescolor
		\begin{tabular}{lc|c|c|c}
			\toprule
            \multirow{2}{*}{\tabincell{c}{ $m$    }} & \multirow{2}{*}{\tabincell{c}{ $n$ }} 
             & {\em s}MBA ($\epsilon = 10^{-5}$) & {\em s}MBA ($\epsilon = 10^{-7}$) & CVX \\
            &	& iter. /  objective / \text{time (s)} 
                & iter. /  objective / \text{time (s)}
                         & objective / \text{time (s)} \\
\midrule
\multirow{3}{*}{\tabincell{c}{ $100$    }}
& 100
  & 667.7 / -25.501867  /   0.6
  &  1675 / -25.506259  /   1.5
   &        -25.506288  /   6.1\\
& 500
  & 841.6 / -134.76308  /   2.2
  &  2573 / -134.76856  /   6.2
   &        -134.76879  /  53.8\\
& 1000
  &  1009 / -272.22067  /   5.7
  &  2577 / -272.22625  /  14.2
   &        -272.22792  / 191.5\\
\midrule
\multirow{3}{*}{\tabincell{c}{ $500$    }}
& 100
  & 972.5 / -29.639329  /  40.1
  &  2260 / -29.640247  /  95.3
   &        -29.640248  / 371.4\\
& 500
  & 692.8 / -139.35545  /  75.0
  &  2265 / -139.36169  / 209.2
   &        -139.36169  / 2253.7\\
& 1000
  & 798.9 / -289.82629  / 152.1
  &  2196 / -289.83148  / 367.5
   &        -289.83151  / 6738.0\\
            			\bottomrule
            		\end{tabular}	
                    }
	\end{table}


}


    \section{\changescolor Concluding remarks}\label{sec6} 
    {\changescolor 
     In this paper, for a large class of conic constrained difference-of-convex optimization problems, we developed a smoothing MBA method whose subproblems always involve one single inequality constraint and studied its convergence behavior. 
    We conclude the paper with further discussions on two technical aspects concerning our algorithm.
    
    {\bf Initialization.} Our algorithm requires an \(x^0 \in G^{-1}(\mathrm{int}\,\mathcal{K})\) as an initial point. Here, we briefly comment on how one may obtain such a point for problems satisfying Assumption~\ref{ass-gen}.

    When \(G\) is $(-{\cal K})$-convex as in Assumption~\ref{ass-conv}, such an $x^0$ can be obtained by first considering an auxiliary problem. For example, one can apply the subgradient method to minimize $\max\{\sigma_{\cal B}(G(\cdot)),-1\}$. Similar preprocessing strategies for obtaining feasible / approximate feasible starting point are commonly used in the literature; see, e.g., \cite[Algorithm~2.1]{CaGoTo14}.
    However, when \(G\) is not $(-{\cal K})$-convex, the aforementioned strategy can get stuck at stationary points of $\max\{\sigma_{\cal B}(G(\cdot)),-1\}$ in general. For these nonconvex problems, one typically needs to rely on the specific problem structure to locate a strictly feasible point. Indeed, even for algorithms that do not maintain feasibility of all iterates like our method, one typically requires an initialization at an $\epsilon$-feasible point to establish a complexity guarantee for an {\em approximately feasible} solution; see, e.g., \cite[Theorem~2]{XieWri21} and \cite[Theorem~3.6]{CaGoTo14}.

    {\bf The prox-broximal operator.} The efficiency of our algorithm relies on the computational complexity of the subproblems  \eqref{opt-sub}, which take the following form:
            \begin{equation}\label{opt-sub-oracle}
				\begin{array}{rl}
					\min\limits_{x\in B(w, R)} &
					P_1(x) + \frac{L}{2}\|x- z\|^2,
				\end{array}
			\end{equation}
    for some $z$, $w\in \mathbb{X}$ and $L > 0$ and $R>0$.
    Similar subproblems have been considered in \cite{16ST, 25GLRR}.
    Specifically, the works \cite{16ST,25GLRR} considered \eqref{opt-sub-oracle} with $L = 0$, and the set of minimizers was named the broximal operator of $P_1$ at $w$ in \cite{25GLRR}. It is then natural to name the unique minimizer of \eqref{opt-sub-oracle} the {\em prox-broximal operator} of $P_1$ at $(z,w)$. It is clear that the prox-broximal operator of $P_1 = 0$ admits a closed form formula, and that of $P_1(\cdot) = \|\cdot\|_1$ can also be computed efficiently and {\em exactly} via a one-dimensional root-finding procedure for solving a {\em piecewise linear quadratic} equation; see \cite[Appendix~A]{21YPL}. Theoretically, when the proximal operator of $\gamma P_1$ can be computed efficiently for all $\gamma > 0$, one can compute the prox-broximal operator via some simple one-dimensional root finding procedures. It is interesting to identify classes of functions (other than the $\ell_1$ norm) whose prox-broximal operators can be computed efficiently and {\em exactly}.
    }


{\changescolor
\section*{Acknowledgments}
The authors thank the two anonymous referees for their comments that helped improve the manuscript. We also thank Shaohua Pan for sharing the codes of iMBA in \cite{25LPB} for our study.
}

\end{document}